\documentclass[12pt,a4paper]{article}

\usepackage[latin1]{inputenc}
\usepackage{amsmath}
\usepackage{amsfonts}
\usepackage{amssymb}
\usepackage{graphicx}
\usepackage{enumerate}
\usepackage{amsthm}
\usepackage{hyperref}

\makeindex
\newtheorem{theorem}{Theorem}[section]
\newtheorem{corollary}[theorem]{Corollary}

\newtheorem{lemma}[theorem]{Lemma}

\newtheorem{proposition}[theorem]{Proposition}

\title{Gibbs States and Gibbsian Specifications on the space $\mathbb{R}^{\mathbb{N}}$}

\author{Artur O. Lopes \thanks{IME - UFRGS. Partially supported by CNPq through project
PQ 310818/2015-0 and CAPES - Finance Code 001} \,and Victor Vargas \thanks{IME - UFRGS. Supported by PNPD-CAPES grant.}}

\begin{document}

\maketitle

\centerline{Abstract}
\smallskip

We are interested in the study of Gibbs and equilbrium probabilities on the lattice $\mathbb{R}^{\mathbb{N}}$.
Consider the unilateral full-shift defined on the non-compact set $\mathbb{R}^{\mathbb{N}}$ and an $\alpha$-H\"older continuous potential $A$ from $\mathbb{R}^{\mathbb{N}}$ into $\mathbb{R}$. From a suitable class of  a priori probability measures $\nu$ (over the Borelian sets of $\mathbb{R}$) we define the Ruelle operator associated to $A$ (using an adequate extension of this operator to the compact set $\overline{\mathbb{R}}^\mathbb{N}=(S^1)^\mathbb{N}$)  and we show the existence of eigenfunctions, conformal probability measures and equilibrium states associated to $A$. The above, can be seen as a generalization of the results obtained in the compact case  for the XY-model.  We also introduce an extension of the definition of entropy  and  show the existence of $A$-maximizing measures. Moreover, we prove the existence of an involution kernel for $A$.  Finally, we build a Gibbsian specification for the Borelian sets on the set $\mathbb{R}^{\mathbb{N}}$ and we show that this family of probability measures satisfies a \emph{FKG}-inequality.

\vspace*{5mm}



\vspace*{5mm}

\maketitle

\section{Introduction}

Thermodynamic formalism is a useful branch of mathematics, mainly in the rigorous study of some interesting problems of statistical mechanics and ergodic theory, which arise from the analysis of physical systems of particles, like molecules of water, noble gases and other type of fluids. These systems usually consist of a large number of elements, commonly of the order of $10^{27}$ elements. In some cases it is required  to study these type of problems when the  lattice is such  that in each site  the set of possible spins is unbounded. When the set of spins is countable several results are already known (see \cite{MR3864383}, \cite{BMP15}, \cite{MR2151222}, \cite{MR2279266}, \cite{MR2800665}, \cite{MR1942310} and section 5 in \cite{MR3377291} )

One of the principal tools used in this area is the Ruelle operator (also called transfer operator), which was introduced initially  by David Ruelle in \cite{MR0234697} as a instrument for the study of compact uni-dimensional lattices. This  has been generalized in several directions, for example, the setting of   \cite{MR2864625} and \cite{MR3377291}  for a compact framework and it was also studied by Mauldin, Urba\'nski and Sarig in \cite{MR2003772} and \cite{MR1738951} in a non-compact setting.

We will study here the case where the lattice is $\mathbb{R}^\mathbb{N}$. The assumptions (on the shift space or in the potential) which are usually considered in most of the cases which analyze the lattice $\mathbb{N}^\mathbb{N}$ (or, $\mathbb{Z}^\mathbb{N}$) are not natural on the present setting.

We will analyze here potentials $A : \mathbb{R}^\mathbb{N}\to \mathbb{R}$ which are Holder with respect to some natural metric. We point out that in order to study continuous potentials on lattices (with a countable set of sites),  such that  the fiber of spins is an unbounded set, some kind of constrictive assumption is needed (for instance, in order that some part of the initial probability  mass do no go to infinity under the dynamical evolution). Here the only technical assumption is in the potential (not related to conditions on  the symbolic space) and is just the Holder assumption on the potential  $A$ in its action on 
$\mathbb{R}^\mathbb{N}$.

Our analysis will be based on the Ruelle operator.

In order to define the Ruelle operator we need an a priori probability.
We will consider the following class of a priori probabilities:
given a continuous function $ f : \mathbb{R} \to \mathbb{R}$ which is  strictly positive and also satisfying the condition $\int_{\mathbb{R}}f(a)da = 1$, we take $d \nu = f(x)\, dx$ as the  priori probability measure

Among some of its utilities, the Ruelle operator allows the construction of Gibbs states, conformal probability measures and \emph{DLR}-Gibbs probability measures through the study of the behavior of its eigenvalues and eigenfunctions. The foregoing provides a good instrument for studying variational problems related with the existence and properties of equilibrium states through properties of linear operators. Nevertheless, the utilities of this operator permeate another important areas of mathematics; for example, W. Parry and M. Pollicott in \cite{MR1085356}  used this operator in order to exhibit important results on thermodynamic formalism of topological Markov chains with some important applications to complex analysis, geometry  and number theory.

In this work we are interested in studying behavior of $\alpha$-H\"older continuous potentials with domain both, in the set of sequences of real numbers and in the set of bi-sequences of real numbers.  This will help on the study of ergodic properties of Gibbs measures for a certain class of potentials using an adaptation of results considered in \cite{MR2864625} and \cite{MR3377291}. Besides that, using some of the features of the Ruelle operator, which is defined from an adequate a priori probability measure, we obtain certain tools that will help to show the existence of calibrated sub-actions, $A$-maximizing probability measures and ground states. Moreover, from properties of this operator, it is obtained a family of probability measures that satisfies the conditions of Gibbsian specification and an \emph{FKG}-inequality, which is an extension of results presented in \cite{MR2864625, MR3690296, CiLo17} which are now extended  (once more) to the  continuous unbounded space of spins. Due to the natural differentiable nature of $\mathbb{R}$ (and,  also - naturally - the one for the product $\mathbb{R}^\mathbb{N})$ we  are able to study some differentiable properties of the eigenfunctions of the Ruelle operator using an involution kernel - we will also consider  an example for  the case of Markov chains.

This paper is organized as follows.

In section \ref{Ruelle-operator-section} we show the existence of Gibbs states and conformal probability measures through properties of the spectrum of the Ruelle operator and its corresponding dual. Furthermore, using the above we build a calibrated sub-action and we present a definition of entropy that extend definitions presented in \cite{MR3377291}, \cite{MR2496111} and \cite{Mohr}. 

Moreover, we show that the given definition of entropy is suitable for the study of  the variational principle of pressure and we use this fact to show the existence of $A$-maximizing measures through a construction related to the existence of ground states associated to $A$. At the end of this section, under mild assumptions, we show that  there exists accumulation points in  the zero temperature limit in our non-compact setting.

In section \ref{Involution-kernel-section} we present  the concept of involution kernel in our setting and we show its existence. Besides that, we build an extension of the Gibbs states to the bilateral case in terms of the involution kernel and the conformal probability measure associated to the potential $A$ and its respective adjunct $A^*$. We also construct an example - for the case of stationary Markov probability measures - where we show some properties of differentiability of the eigenfunctions associated to the Ruelle operator in the case of locally constant potentials.

In section \ref{FKG-inequalities-section} we build a family of probability measures that generates a Gibbsian specification for the Borelian sets on $\mathbb{R}^{\mathbb{N}}$ and we demonstrate that this family of probability measures satisfies an \emph{FKG}-inequality under certain assumptions for the potential $A$. Besides that, from the classical approach of thermodynamic limit, we show that any thermodynamic limit of  probability measures belonging to the Gibbsian specification is in fact a \emph{DLR}-Gibbs probability measure and we also show that the Gibbs state and the conformal probability measure constructed in the section \ref{Ruelle-operator-section} are  \emph{DLR}-Gibbs probability measures.

\section{Ruelle operator}
\label{Ruelle-operator-section}

The Ruelle operator is one of the main tools used in Thermodynamical Formalism. This operator allows the construction of equilibrium states through an algebraic approach, which helps a lot  the study of variational principles for the topological pressure (among other important features). With this aim in this section we will build a Ruelle operator from a suitable a priori probability measure.

Let $a, b \in \mathbb{R}$, note that the map $(a, b) \mapsto \frac{1}{\pi} |\arctan(a)  - \arctan(b)|$ defines a metric on $\mathbb{R}$.

Points $x$ in $\mathbb{R}^\mathbb{N}$ are denoted by $x=(x_1,x_2,..,x_n,..).$

We consider on $\mathbb{R}^\mathbb{N}$ the metric

 \[
d(x, y) = \sum_{n \in \mathbb{N}} \frac{1}{\pi2^n} |\arctan(x_n) - \arctan(y_n)| \,.
\]

We will analyze here potentials $A : \mathbb{R}^\mathbb{N}\to \mathbb{R}$ which are Holder with respect to such metric.

By defining $\arctan(\pm\infty) = \lim_{x \to \pm\infty} \arctan(x)$ one can get  an extension of this metric to the set of extended real numbers defined as the one point compactification $\overline{\mathbb{R}} = \mathbb{R} \cup \{\pm \infty\}$. Moreover, the set $\overline{\mathbb{R}}$ equipped with this metric is a compact metric space. Therefore, follows from the Tychonoff's Theorem that $\overline{\mathbb{R}}^{\mathbb{N}} = \{(x_n)_{n \in \mathbb{N}}: x_n \in \overline{\mathbb{R}}, \forall n\}$ is a compact metric space equipped with the metric
\[
\widehat{d}(x, y) = \sum_{n \in \mathbb{N}} \frac{1}{\pi2^n} |\arctan(x_n)- \arctan(y_n)|\,.
\]

Indeed, note that for any basic open set in the product topology, which is of the form
\[
[U_1 \ldots U_k] = \prod_{n \in \mathbb{N}} U_n \,,
\]

with $U_n = \mathbb{R}$ for any $n > k$ and $U_1, \ldots, U_k$ open sets in $\mathbb{R}$, we get  that for each $x \in [U_1 \ldots U_k]$, there is $\epsilon_0 > 0$ such that

\[
B_{\widehat{d}}(x; \epsilon_0) \subset [U_1 \ldots U_k] \subset \bigcup_{x \in [U_1 \ldots U_k]} B_{\widehat{d}}(x; \epsilon_0) \,.
\]

Hereafter, we will denote by $\mathcal{C}_b(X)$ the set of bounded continuous functions from $X$ into $\mathbb{R}$. Consider the Lebesgue measure $d x$ on the Borel sigma-algebra of $\mathbb{R}$.

Fixing $f : \mathbb{R} \to \mathbb{R}$ an strictly positive continuous function satisfying $\int_{\mathbb{R}}f(a)da = 1$ and choosing $\nu = fdx$ as a priori probability measure, we define the Ruelle operator $\mathcal{L}_A$ associated to an $\alpha$-H\"older continuous potential $A : \mathbb{R}^{\mathbb{N}} \to \mathbb{R}$ as the map assigning to each $\varphi \in \mathcal{C}_b(\mathbb{R}^{\mathbb{N}})$ the function

\[
\mathcal{L}_A (\varphi) (x) = \int_{\mathbb{R}} e^{A(ax)}\varphi(ax)d\nu(a) = \int_{\mathbb{R}} e^{A(ax)}\varphi(ax)f(a)da \,.
\]

By the above definition, follows that for with $a^n = (a_n, \dots, a_1)$ and $S_nA(x) = \sum^{n-1}_{k=0}A(\sigma^n(x))$, the $n$-th iterative of the Ruelle operator is given by
\[
\mathcal{L}^n_A (\varphi) (x) = \int_{\mathbb{R}^n} e^{S_nA(a^nx)}\varphi(a^nx)f(a_n)\ldots f(a_1)da_1\ldots da_n \,.
\]

Furthermore, using that $\mathcal{L}_A$ is an operator from $\mathcal{C}_b(\mathbb{R}^{\mathbb{N}})$ into $\mathcal{C}_b(\mathbb{R}^{\mathbb{N}})$, it is possible to define the dual of the Ruelle operator, as the map from the set of regular additive finite Borel measures into itself, satisfying for any $\varphi \in \mathcal{C}_b(\mathbb{R}^{\mathbb{N}})$ the following equation
\[
\int_{\mathbb{R}^{\mathbb{N}}} \varphi d\left(\mathcal{L}^*_{A}\mu\right) = \int_{\mathbb{R}^{\mathbb{N}}} \mathcal{L}_{A}(\varphi) d\mu \,.
\]

Given a metric space $(X, d)$ we denote by $\mathcal{H}_{\alpha}(X)$ the set of $\alpha$-H\"older continuous functions from $X$ into $\mathbb{R}$. As usual, we will denote the H\"older constant of $A \in \mathcal{H}_{\alpha}(X)$ as $\mathrm{Hol}_{A} = \sup_{x \neq y}\frac{|A(x) - A(y)|}{d(x, y)^{\alpha}}$. It is widely known that $\mathcal{H}_{\alpha}(X)$ equipped with the norm $\|A\|_{\alpha} = \|A\|_{\infty} + \mathrm{Hol}_A$ is a Banach space. Besides that, given a continuous map $T : X \to X$ we are going to denote by $\mathcal{M}_T(X)$ the set of $T$-invariant Borel probability measures on $X$.

Following a similar procedure that in Proposition 4 of \cite{MR2864625}, it is possible showing that the image of any $\alpha$-H\"older continuous function by the Ruelle operator it is also $\alpha$-H\"older continuous, in other words, the restriction of $\mathcal{L}_A$ to $\mathcal{H}_{\alpha}(\mathbb{R}^{\mathbb{N}})$ is a map into $\mathcal{H}_{\alpha}(\mathbb{R}^{\mathbb{N}})$.

The following Lemma provides a useful tool which will help us  in getting most of properties in the case that we are studying in this work.

\begin{lemma}
Suppose that $A \in \mathcal{H}_{\alpha}(\mathbb{R}^{\mathbb{N}})$, then $A$ can be extended to a unique function $A' \in \mathcal{H}_{\alpha}(\overline{\mathbb{R}}^{\mathbb{N}})$.
\label{Holder}
\end{lemma}

\begin{proof}
For any $x \in \overline{\mathbb{R}}^{\mathbb{N}}$ we define $A'(x) = \lim_{y \to x}A(y)$. Note that this limit exists because $A \in \mathcal{H}_{\alpha}(\mathbb{R}^{\mathbb{N}})$, moreover, it is finite since $A$ is a bounded potential. Then, for $x, y \in \overline{\mathbb{R}}^{\mathbb{N}}$ there exist sequences $(x^n)_{n \in \mathbb{N}}$ and $(y^m)_{m \in \mathbb{N}}$ taking values in $\mathbb{R}^{\mathbb{N}}$ such that $\lim_{n \in \mathbb{N}} x^n = x$ and $\lim_{m \in \mathbb{N}} z^m = z$, which implies
\[
|A'(x) - A'(y)| = \lim_{n\in\mathbb{N}}\lim_{m\in\mathbb{N}}|A(x^n) - A(x^m)| \leq \mathrm{Hol}_A\widehat{d}(x, y) \,.
\]

The above implies that $A' \in \mathcal{H}_{\alpha}(\overline{\mathbb{R}^{\mathbb{N}}})$. The uniqueness of $A'$ follows of the properties of limits using the perfectness of $\overline{\mathbb{R}^{\mathbb{N}}}$.
\end{proof}

\begin{corollary}
$\mathcal{H}_{\alpha}(\mathbb{R}^{\mathbb{N}})$ is isometrically isomorphic to $\mathcal{H}_{\alpha}(\overline{\mathbb{R}^{\mathbb{N}}})$.
\label{isometrically-isomorphic}
\end{corollary}

\begin{proof}
	We first observe that for each $A\in \mathcal{H}_{\alpha}(\mathbb{R}^{\mathbb{N}})$ we have that $\|A\|_{\infty} = \|A'\|_{\infty}$ and $\mathrm{Hol}_{A} =\mathrm{Hol}_{A'}$. Therefore, the identity map $I:\mathcal{H}_{\alpha}(\mathbb{R}^{\mathbb{N}})\to \mathcal{H}_{\alpha}(\overline{\mathbb{R}^{\mathbb{N}}})$
is a linear isometry and therefore an injective map. From the uniqueness and continuity of the extension provided by Lemma \ref{Holder} it follows
that $I$ is onto. Indeed, if $A\in \mathcal{H}_{\alpha}(\overline{\mathbb{R}^{\mathbb{N}}})$ then $A = (A|_{\mathbb{R}^{\mathbb{N}}})' = I(A|_{\mathbb{R}^{\mathbb{N}}})$.
\end{proof}

Using the Lemma \ref{Holder} and taking $B \in \mathcal{H}_{\alpha}(\overline{\mathbb{R}}^{\mathbb{N}})$, we can extend the definition of Ruelle operator to $\mathcal{H}_{\alpha}(\overline{\mathbb{R}}^{\mathbb{N}})$ in a natural way as the map assigning to each $\varphi \in \mathcal{H}_{\alpha}(\overline{\mathbb{R}}^{\mathbb{N}})$ the function
\[
\mathcal{L}_B (\varphi) (x) = \int_{\overline{\mathbb{R}}} e^{B(ax)}\varphi(ax)d\nu(a) = \int_{\mathbb{R}} e^{B(ax)}\varphi(ax)f'(a)da \,.
\]

Moreover in the following Lemma we will show that $\mathcal{L}_A$ is topologically conjugated to $\mathcal{L}_{A'}$ through the isometry defined in the Corollary \ref{isometrically-isomorphic}.

\begin{lemma}
Consider $A\in \mathcal{H}_{\alpha}(\mathbb{R}^{\mathbb{N}})$. Then $\mathcal{L}_{A'} \circ I = I \circ \mathcal{L}_{A}$, where $I : \mathcal{H}_{\alpha}(\mathbb{R}^{\mathbb{N}}) \to \mathcal{H}_{\alpha}(\overline{\mathbb{R}^{\mathbb{N}}})$ is the isometry provided by Corollary \ref{isometrically-isomorphic}.
\label{topological-conjugation}
\end{lemma}

\begin{proof}
For any  $\varphi \in \mathcal{H}_{\alpha}(\mathbb{R}^{\mathbb{N}})$, and each $x \in \mathbb{R}^{\mathbb{N}}$ we have we have
\begin{align}
(I^{-1} \circ \mathcal{L}_{A'} \circ I)(\varphi)(x)
&= I^{-1} (\mathcal{L}_{A'} \circ I)(\varphi)(x) \nonumber \\
&= I^{-1} (\mathcal{L}_{A'}(I(\varphi)))(x) \nonumber \\
&= I^{-1} (\mathcal{L}_{A'}(\varphi'))(x) \nonumber \\
&= I^{-1} (\mathcal{L}_{A}(\varphi))(x) \nonumber \\
&= \mathcal{L}_{A}(\varphi)(x) \nonumber \,.
\end{align}

Note that the second last equality is a consequence of the assumption $\lim_{a \to \pm \infty}f(a)da = 0$ and the boundedness of $A$, which implies that
\[
0 \leq \left| \lim_{R \to +\infty}\int_{(-R, R)^c} e^{A'(ax)}\varphi'(ax)f'(a)da \right| \leq e^{\|A\|_{\infty}}\|\varphi\|_{\infty} \lim_{R \to +\infty}\int_{(-R, R)^c} f'(a)da = 0 \,.
\]
\end{proof}

Observe that compactness of $\overline{\mathbb{R}}^{\mathbb{N}}$ with the metric $\widehat{d}$, implies that results demonstrated in \cite{MR3377291} are valid for $A' \in \mathcal{H}_{\alpha}(\overline{\mathbb{R}}^{\mathbb{N}})$.

\begin{theorem}
Consider $A \in \mathcal{H}_{\alpha}(\mathbb{R}^{\mathbb{N}})$, then:
\begin{enumerate}[a)]
\item There is $\lambda_A > 0$, and $\psi_A : \mathbb{R}^{\mathbb{N}} \to \mathbb{R}$ an strictly positive $\alpha$-H\"older continuous function, such that $\mathcal{L}_A (\psi_A) (x) = \lambda_A \psi_A(x)$, for all $x \in \mathbb{R}^{\mathbb{N}}$. Moreover, the main eigenvalue is simple.
\item Defining $\bar{A} = A + \log(\psi_A) - \log(\psi_A \circ \sigma) - \log(\lambda_A)$, there is a unique fixed point $\mu_A$ of $\mathcal{L}^*_{\bar{A}}$ which belongs to $\mathcal{M}_\sigma(\mathbb{R}^{\mathbb{N}})$. This probability measure is called Gibbs state associated for $A$.
\item Choosing adequately the eigenfunction $\psi_A$, we have that $\rho_A = \frac{1}{\psi_A}d\mu_A$ is a probability measure satisfying $\mathcal{L}^*_{A}\rho_A = \lambda_A\rho_A$. This measure is called conformal probability measure for $A$.
\item For any $w \in \mathcal{H}_{\alpha}(\mathbb{R}^{\mathbb{N}})$ we have existence of the following uniform limits:
\[
\lim_{n \in \mathbb{N}} \mathcal{L}^n_{\bar{A}}w = \int_{\mathbb{R}^{\mathbb{N}}}w d\mu_A \,.
\]

and
\[
\lim_{n \in \mathbb{N}} \frac{\mathcal{L}^n_{A}w}{\lambda^n_A} = \psi_A\int_{\mathbb{R}^{\mathbb{N}}}w d\rho_A \,.
\]
\end{enumerate}
\label{Ruelle}
\end{theorem}

\begin{proof}
By Lemma \ref{Holder}, there is an $\alpha$-H\"older continuous extension $A'$ of $A$ defined on the set of extended real numbers $\overline{\mathbb{R}}^{\mathbb{N}}$. Moreover, by Theorem $2$ in \cite{MR3377291} the items $a)$, $b)$, $c)$ and $d)$ of this Theorem are valid for $A'$.

Now, taking $\lambda_A = \lambda_{A'}$ and $\psi_A = \psi_{A'}|_{\mathbb{R}^{\mathbb{N}}}$, we obtain item $a)$ of this Theorem. Indeed, by Lemma \ref{topological-conjugation} we have $\mathcal{L}_{A'} \circ I = I \circ \mathcal{L}_A$, which implies that
\begin{align}
\mathcal{L}_A(\psi_A)
&= (I^{-1} \circ \mathcal{L}_{A'} \circ I)(\psi_A) \nonumber \\
&= (I^{-1} \circ \mathcal{L}_{A'} \circ I)(\psi_{A'}|_{\mathbb{R}^{\mathbb{N}}}) \nonumber \\
&= (I^{-1} \circ \mathcal{L}_{A'})(\psi_{A'}) \nonumber \\
&= \lambda_{A'}I^{-1}(\psi_{A'}) \nonumber \\
&= \lambda_{A'}\psi_A \nonumber \,.
\end{align}

Therefore, in the process of showing  the remaining items, it is enough to show that
\begin{equation}
\mu_{A'}(\overline{\mathbb{R}}^{\mathbb{N}} \setminus \mathbb{R}^{\mathbb{N}}) = 0 \,.
\label{tail}
\end{equation}

Indeed, under the assumption that (\ref{tail}) is valid, we can define $\mu_A = \mu_{A'}|_{\mathbb{R}^{\mathbb{N}}}$, in other words, the probability measure assigning to any Borelian set $E' \subset \overline{\mathbb{R}^{\mathbb{N}}}$, the value
\[
\mu_{A'}|_{\mathbb{R}^{\mathbb{N}}}(E') = \frac{\mu_{A'}(E' \cap \mathbb{R}^{\mathbb{N}})}{\mu_{A'}({\mathbb{R}^{\mathbb{N}}})} = \mu_{A'}(E' \cap \mathbb{R}^{\mathbb{N}}) \,.
\]

Therefore $\mu_A(E) = \mu_{A'}(E \cap \mathbb{R}^{\mathbb{N}}) = \mu_{A'}(E)$ for any Borelian set $E \subset \mathbb{R}^{\mathbb{N}}$, which implies that $\mu_A$ satisfies the items $b)$, $c)$ and $d)$ of this Theorem.

Now, observe that $\mu_{A'}(\{x \in \overline{\mathbb{R}}^{\mathbb{N}} : x_1 = \pm \infty\}) = 0$ implies (\ref{tail}). Now it follows from  invariance of $\mu_{A'}$ regarding the map $\sigma$ and the above condition  that $\mu_{A'}(\{x \in \overline{\mathbb{R}}^{\mathbb{N}} : x_n = \pm \infty\}) = 0$, for all $n \in \mathbb{N}$. Therefore,
\[
\mu_{A'}(\overline{\mathbb{R}}^{\mathbb{N}} \setminus \mathbb{R}^{\mathbb{N}}) \leq \sum_{n \in \mathbb{N}} \mu_{A'}(\{x \in \overline{\mathbb{R}}^{\mathbb{N}} : x_n = \pm \infty\}) = 0 \,.
\]

In order to demonstrate that $\mu_{A'}(\{x \in \overline{\mathbb{R}}^{\mathbb{N}} : x_1 = \pm \infty\}) = 0$ we will use the following procedure.

Fixing $R > 0$, we can choose an $\alpha$-H\"older continuous function $w_R : \overline{\mathbb{R}}^{\mathbb{N}} \to [0, 1]$ that satisfies the following conditions:
\begin{enumerate}[i)]
\item $\chi_{\{x \in \overline{\mathbb{R}}^{\mathbb{N}} : x_1 = \pm \infty\}}(x) \leq w_R(x) \leq 1$.
\item $w_R(x) = 0$, if $x \in (-R, R)$.
\end{enumerate}
Then, we have
\begin{align}
&\mu_{A'}(\{x \in \overline{\mathbb{R}}^{\mathbb{N}} : x_1 = \pm \infty\}) \nonumber \\
&= \int_{\overline{\mathbb{R}}^{\mathbb{N}}}\chi_{\{x \in \overline{\mathbb{R}}^{\mathbb{N}} : x_1 = \pm \infty\}} d\mu_{A'} \nonumber \\
&\leq \int_{\overline{\mathbb{R}}^{\mathbb{N}}} w_R d\mu_{A'} \nonumber \\
&= \lim_{n \in \mathbb{N}} \mathcal{L}^{n}_{\overline{A'}} (w_R)(x) \nonumber \\
&= \lim_{n \in \mathbb{N}} \int_{\overline{\mathbb{R}}^n} e^{S_n\overline{A'}(a^nx)}w_R(a^nx)f(a_n)\ldots f(a_1)da_1\ldots da_n \nonumber \\
&= \lim_{n \in \mathbb{N}} \int_{(-R, R)^c} e^{\overline{A'}(a_nx)} \nonumber \\
& \ \ \ \ \left(\int_{\overline{\mathbb{R}}^{n-1}} e^{S_{n-1}\overline{A'}(a^{n-1}x)}f(a_{n-1})\ldots f(a_1)da_1\ldots da_{n-1}\right)f(a_n)da_n \nonumber \\
&= \lim_{n \in \mathbb{N}} \int_{(-R, R)^c} e^{\overline{A'}(a_nx)}f(a_n)da_n \leq e^{\|\overline{A'}\|}\int_{(-R, R)^c} f(a)da \,. \nonumber
\end{align}
Taking the limit when $R \to +\infty$ we get our claim.
\end{proof}

Note that under the hypothesis of the above Theorem, for any $\beta > 0$ we have that $\lambda_{\beta A}$, $\psi_{\beta A}$, $\mu_{\beta A}$, and $\rho_{\beta A}$ are well defined, and satisfy the same conclusions of the Theorem. The following Lemma shows a bound for the family of logarithms of the eigenvalues associated to the family $(\beta A)_{\beta > 0}$. This bound will be useful in the proofs of the results that appear below.

\begin{lemma}
If $A \in \mathcal{H}_{\alpha}(\mathbb{R}^{\mathbb{N}})$, then for any $\beta > 0$ we have
\[
-\|A\| \leq \frac{1}{\beta}\log(\lambda_{\beta A}) \leq \|A\| \,.
\]
\end{lemma}

\begin{proof}
Let $A' : \overline{\mathbb{R}}^{\mathbb{N}} \to \mathbb{R}$ be the $\alpha$-H\"older continuous extension of $A$, then $\inf(\psi_{\beta A}) = \min(\psi_{\beta A'})$, and $\sup(\psi_{\beta A}) = \max(\psi_{\beta A'})$. The foregoing implies that for any $n \in \mathbb{N}$ there exists $\tilde{x}^n \in \mathbb{R}^{\mathbb{N}}$, such that $\psi_{\beta A}(\tilde{x}^n) < \inf(\psi_{\beta A}) + 1/n$. Therefore, using that $\mathcal{L}_{\beta A}(\psi_{\beta A})(\tilde{x}^n) = \lambda_{\beta A}\psi_{\beta A}(\tilde{x}^n)$, follows
\begin{align}
\lambda_{\beta A}
= \frac{1}{\psi_{\beta A}(\tilde{x}^n)}\mathcal{L}_{\beta A}(\psi_{\beta A})(\tilde{x}^n)
&= \frac{1}{\psi_{\beta A}(\tilde{x}^n)}\int_{\mathbb{R}} e^{\beta A(a\tilde{x}^n)}\psi_{\beta A} (a\tilde{x}^n)f(a)da \nonumber \\
&> \frac{1}{\inf(\psi_{\beta A}) + 1/n}\int_{\mathbb{R}} e^{\beta A(a\tilde{x}^n)}\psi_{\beta A} (a\tilde{x}^n)f(a)da \nonumber \\
&> \frac{1}{\inf(\psi_{\beta A})}\int_{\mathbb{R}} e^{\beta A(a\tilde{x}^n)}\psi_{\beta A} (a\tilde{x}^n)f(a)da \nonumber \,.
\end{align}
Therefore, we have
\[
\lambda_{\beta A} \geq \int_{\mathbb{R}} e^{\beta A(a\tilde{x}^n)} f(a)da \geq e^{-\beta\|A\|} \,.
\]

Using a similar argument, it is demonstrated that $\lambda_{\beta A} \leq e^{\beta\|A\|}$, which concludes the proof.
\end{proof}

One of the main problems in non-compact setting  is existence or not of maximizing measures. Nevertheless, the next Theorem provides conditions in which there exists  a maximizing measure for an $\alpha$-H\"older continuous potential.

\begin{theorem}
Set $A \in \mathcal{H}_{\alpha}(\mathbb{R}^{\mathbb{N}})$. If there exists $z_0 \in \mathbb{R}$ such that the extension $A' \in \mathcal{H}_{\alpha}(\overline{\mathbb{R}}^{\mathbb{N}})$ satisfies
\begin{equation}
A'(x_1, \ldots, x_{n-1}, \pm\infty, x_{n+1}, \ldots) < A'(x_1, \ldots, x_{n-1}, z_0, x_{n+1}, \ldots) \,,
\label{decay}
\end{equation}

for all $x \in \overline{\mathbb{R}}^{\mathbb{N}}$ and all $n \in \mathbb{N}$, then:
\begin{enumerate}[a)]
\item Any maximizing measure $\mu_{\infty}$ of $A'$ has support contained in $\mathbb{R}^{\mathbb{N}}$, therefore $\mu_{\infty}$ is a maximizing measure of $A$.
\item $A$ has a calibrated sub-action $V$ defined on $\mathbb{R}^{\mathbb{N}}$, in other words, $V$ is a continuous function that satisfies
\[
m(A) = \max_{a \in \mathbb{R}}\{A(ax) + V(ax) - V(x)\} \,.
\]

\item $\lim_{\beta \to \infty}\frac{1}{\beta}\log(\lambda_{\beta A}) = m(A)$.
\end{enumerate}
\label{calibrated-sub-action}
\end{theorem}

\begin{proof}
Consider $A'$ the $\alpha$-H\"older continuous extension of $A$ to the set $\overline{\mathbb{R}}^{\mathbb{N}}$, using that $\frac{1}{\beta}\log(\psi_{\beta A'})$ is $\alpha$-H\"older continuous with constant $\frac{2^{\alpha}}{2^{\alpha}-1}\mathrm{Hol}_A$ for all $\beta > 0$, follows that the family $(\frac{1}{\beta}\log(\psi_{\beta A'}))_{\beta>0}$ is equi-continuous and uniformly bounded, therefore, by Arzela-Ascoli's Theorem, there exists a convergent sub-sequence $(\frac{1}{\beta_n}\log(\psi_{\beta_n A'}))_{n \in \mathbb{N}}$ and, by Proposition 10 in \cite{MR3377291}, $V' = \lim_{n \in \mathbb{N}}\frac{1}{\beta_n}\log(\psi_{\beta_n A'})$ is a calibrated sub-action for $A'$, in other words, for all $x \in \overline{\mathbb{R}}^{\mathbb{N}}$ is satisfied
\[
m(A') = \max_{a\in\overline{\mathbb{R}}}\{A'(ax) + V'(ax) - V'(x)\} \,.
\]

Moreover, a similar result in \cite{MR3377291}, guarantees that $\lim_{\beta \to \infty}\frac{1}{\beta}\log(\lambda_{\beta A'}) = m(A')$, therefore $\lim_{\beta \to \infty}\frac{1}{\beta}\log(\lambda_{\beta A}) = m(A')$.

Now, using (\ref{decay}), we obtain the following inequalities for each $x \in \overline{\mathbb{R}}^{\mathbb{N}}$
\begin{enumerate}[i)]
\item $\psi_{\beta A'}(\pm \infty, x_1, x_2, \ldots) \leq \psi_{\beta A'}(z_0, x_1, x_2, \ldots)$.
\item $V'(\pm \infty, x_1, x_2, \ldots) \leq V'(z_0, x_1, x_2, \ldots)$
\end{enumerate}
Indeed, since for all $x \in \overline{\mathbb{R}}^{\mathbb{N}}$
\begin{equation}
\rho_{\beta A'}(\overline{\mathbb{R}}^{\mathbb{N}})\psi_{\beta A'}(x) = \lim_{n \in \mathbb{N}}\int_{\overline{\mathbb{R}}^n} e^{S_n\beta A'(a^nx) - n\log(\lambda_{\beta A'})}f(a_n)\ldots f(a_1)da_1 \ldots da_n \,.
\label{limit}
\end{equation}

and the points $y = (\pm\infty x)$ and $z = (z_0x)$ satisfy for each $n \in \mathbb{N}$ the condition
\[
\int_{\overline{\mathbb{R}}^n} e^{S_n\beta A'(a^n y)}f(a_n)\ldots f(a_1)da_1 \ldots da_n < \int_{\overline{\mathbb{R}}^n} e^{S_n\beta A'(a^n z)}f(a_n)\ldots f(a_1)da_1 \ldots da_n \,.
\]

Then, it follows from (\ref{limit}), that $\psi_{\beta A'}(y) \leq \psi_{\beta A'}(z)$, therefore using that $V' = \lim_{n \in \mathbb{N}}\frac{1}{\beta_n}\log(\psi_{\beta_n A'})$, we obtain $V'(y) \leq V'(z)$. The last inequality joint with (\ref{decay}) implies that for all $x \in \overline{\mathbb{R}}^{\mathbb{N}}$
\[
A'(\pm\infty x) + V'(\pm\infty x) - V'(x) < A'(z_0 x) + V'(z_0 x) - V'(x) \,.
\]

Therefore, we have
\[
\max_{a\in\overline{\mathbb{R}}}\{A'(ax) + V'(ax) - V'(x)\} = \max_{a\in\mathbb{R}}\{A'(ax) + V'(ax) - V'(x)\} \,.
\]

In other words, using that $A = A'|_{\mathbb{R}^{\mathbb{N}}}$, and taking $V = V'|_{\mathbb{R}^{\mathbb{N}}}$, it follows from the properties of the calibrated sub-action of $V'$, that for all $x \in \mathbb{R}^{\mathbb{N}}$
\[
m(A') = \max_{a\in\mathbb{R}}\{A(ax) + V(ax) - V(x)\} \,.
\]

Thus, in the process of finalizing the proof, we only need to demonstrate that $m(A) = m(A')$, and for that, it is enough showing that any $A'$-maximizing measure it is supported in $\mathbb{R}^{\mathbb{N}}$.

Indeed, observe that any $A'$-maximizing probability measure $\mu_{\infty}$ satisfies
\[
\int_{\overline{\mathbb{R}}^{\mathbb{N}}}(A' + V' - V' \circ \sigma - m(A'))d\mu_{\infty} = 0 \,,
\]
then, taking in account that $A' + V' - V' \circ \sigma - m(A')$ is less than or equal to zero and continuous, it follows that this function vanishes at the support of $\mu_{\infty}$.

Now observe that, for $x \in \overline{\mathbb{R}}^{\mathbb{N}} \setminus \mathbb{R}^{\mathbb{N}}$ there is $k = k(x) \in \mathbb{N}$, such that, $x_k = \pm\infty$, and $x_l \neq \pm\infty$ for $1 \leq l < k$, then taking $y$ defined by $y_i = x_i$ for all $i \neq k$ and $y_k = z_0$, it follows that
\begin{align}
&A'(\sigma^{k-1}(x)) + V'(\sigma^{k-1}(x)) - V'(\sigma^k(x)) - m(A') \nonumber \\
&< A'(\sigma^{k-1}(y)) + V'(\sigma^{k-1}(y)) - V'(\sigma^k(y)) - m(A') \leq 0\nonumber \,,
\end{align}

which implies that $x$ does not belongs to the support of any maximizing measure.
\end{proof}

From now on, we will denote the set of Gibbs states associated to the a priori probability measure $\nu = fdx$ as $\mathcal{G}$, in other words, the set of $\mu \in \mathcal{M}_{\sigma}(\mathbb{R}^{\mathbb{N}})$, such that, $\mathcal{L}^*_B(\mu) = \mu$ for some $\alpha$-H\"older normalized potential $B
$. In this case we define the entropy of $\mu \in \mathcal{G}$ as
\begin{equation}
h(\mu) = -\int_{\mathbb{R}^{\mathbb{N}}}B d\mu \,.
\label{entropy-Gibbs}
\end{equation}

In particular, when $B = \bar{A}$ for some $\alpha$-H\"older continuous potential $A$, we have
\begin{align}
h(\mu_A)
&= -\int_{\mathbb{R}^{\mathbb{N}}}\bar{A} d\mu_A \nonumber \\
&= -\int_{\mathbb{R}^{\mathbb{N}}}A + \log(\psi_A) - \log(\psi_A \circ \sigma) - \log(\lambda_A)d\mu_A \nonumber \\
&= -\int_{\mathbb{R}^{\mathbb{N}}}A d\mu_A + \log(\lambda_A) \nonumber \,.
\end{align}

We can extend (\ref{entropy-Gibbs}) for the set $\mathcal{M}_{\sigma}(\mathbb{R}^{\mathbb{N}})$ in the following way
\begin{equation}
h(\mu) = \inf_{u \in \mathcal{C}^+_b(\mathbb{R}^{\mathbb{N}})}\left\{ \int_{\mathbb{R}^{\mathbb{N}}} \log\left(\frac{\mathcal{L}_0(u)}{u}\right) d\mu \right\} \,.
\label{entropy}
\end{equation}

With $\mathcal{C}^+_b(\mathbb{R}^{\mathbb{N}})$ the set of strictly positive bounded continuous functions from $\mathbb{R}^{\mathbb{N}}$ into $\mathbb{R}$.

Henceforth we are going to show that the (\ref{entropy}) coincides with (\ref{entropy-Gibbs}) in the case that $\mu \in \mathcal{G}$, furthermore, we are going to show that this definition in fact satisfies a variational principle.

\begin{lemma}
Let $\mu \in \mathcal{G}$ a Gibbs state associated to a normalized potential $B$. Then
\[
h(\mu) = \inf_{u \in \mathcal{C}^+_b(\mathbb{R}^{\mathbb{N}})}\left\{ \int_{\mathbb{R}^{\mathbb{N}}} \log\left(\frac{\mathcal{L}_0(u)}{u}\right) d\mu \right\} \,.
\]
\end{lemma}

\begin{proof}
Set $u_0 = e^B$, then $u_0$ belongs to $\mathcal{C}^+_b(\mathbb{R}^{\mathbb{N}})$, moreover, since $B$ is a normalized potential, it follows that
\[
\log\left( \frac{\mathcal{L}_0(u_0)}{u_0} \right) = -B \,.
\]

Therefore, integrating both of the sides of the equation regarding the measure $\mu$, we obtain that
\[
\int_{\mathbb{R}^{\mathbb{N}}}\log\left( \frac{\mathcal{L}_0(u_0)}{u_0} \right)d\mu = h(\mu) \,.
\]

On the other hand, for any $\overline{u} \in \mathcal{C}^+_b(\mathbb{R}^{\mathbb{N}})$ and  the function $u = \overline{u}e^{-B} \in \mathcal{C}^+_b(\mathbb{R}^{\mathbb{N}})$ we get $\mathcal{L}_0(\overline{u}) = \mathcal{L}_B(u)$, which implies that
\[
\log\left( \frac{\mathcal{L}_0(\overline{u})}{\overline{u}} \right) = \log(\mathcal{L}_B(u)) - \log(u) - B \,.
\]

Therefore, integrating in both of the sides of the equation regarding the measure $\mu$, we obtain that
\[
\int_{\mathbb{R}^{\mathbb{N}}}\log\left( \frac{\mathcal{L}_0(\overline{u})}{\overline{u}} \right) d\mu = \int_{\mathbb{R}^{\mathbb{N}}}\log(\mathcal{L}_B(u))d\mu - \int_{\mathbb{R}^{\mathbb{N}}}\log(u)d\mu + h(\mu) \,.
\]

Then, in order to demonstrate this Lemma, it is enough showing that
\[
\int_{\mathbb{R}^{\mathbb{N}}}\log(\mathcal{L}_B(u))d\mu - \int_{\mathbb{R}^{\mathbb{N}}}\log(u)d\mu \geq 0 \,.
\]

Indeed, using that $B$ is a normalized potential, by Jensen's inequality we obtain that $\int_{\mathbb{R}^{\mathbb{N}}}\log(\mathcal{L}_B(u))d\mu \geq \int_{\mathbb{R}^{\mathbb{N}}} \mathcal{L}_B(\log(u))d\mu$, therefore, we obtain the desired inequality.
\end{proof}

Note that the above Lemma shows that (\ref{entropy}) extends (\ref{entropy-Gibbs}) to the set of $\mathcal{M}_{\sigma}(\mathbb{R}^{\mathbb{N}})$. Moreover for any $\mu \in \mathcal{M}_{\sigma}(\mathbb{R}^{\mathbb{N}})$ we have $h(\mu) \leq 0$. Indeed, it follows by (\ref{entropy}) that
\[
h(\mu) \leq \int_{\mathbb{R}^{\mathbb{N}}}\log\left(\mathcal{L}_0(1)(x)\right)d\mu(x) = \int_{\mathbb{R}^{\mathbb{N}}}\log\left(\int_{\mathbb{R}}f(a)da\right)d\mu(x) = 0 \,.
\]

Besides that, for any $A \in \mathcal{H}_{\alpha}(\mathbb{R}^{\mathbb{N}})$ and each $\mu \in \mathcal{M}_{\sigma}(\mathbb{R}^{\mathbb{N}})$, we have

\begin{align}
h(\mu) + \int_{\mathbb{R}^{\mathbb{N}}} Ad\mu
&= \inf_{u \in \mathcal{C}^+_b(\mathbb{R}^{\mathbb{N}})}\left\{ \int_{\mathbb{R}^{\mathbb{N}}} \log\left(\frac{\mathcal{L}_0(u)}{u}\right) d\mu\right\} + \int_{\mathbb{R}^{\mathbb{N}}} Ad\mu \nonumber \\
&\leq \int_{\mathbb{R}^{\mathbb{N}}} \log\left(\frac{\mathcal{L}_0(e^A\psi_A)}{e^{A}\psi_A}\right) d\mu + \int_{\mathbb{R}^{\mathbb{N}}} \log(e^{A})d\mu \nonumber \\
&= \int_{\mathbb{R}^{\mathbb{N}}} \log\left(\frac{\mathcal{L}_0(e^A\psi_A)}{\psi_A}\right) d\mu \nonumber \\
&= \int_{\mathbb{R}^{\mathbb{N}}} \log\left(\frac{\mathcal{L}_A(\psi_A)}{\psi_A}\right) d\mu \nonumber \\
&= \log(\lambda_A) \nonumber \,.
\end{align}

The next theorem shows that (\ref{entropy}) satisfies a variational principle.

\begin{theorem}
Consider $A \in \mathcal{H}_{\alpha}(\mathbb{R}^{\mathbb{N}})$ and $\lambda_A$ be the maximal eigenvalue of $\mathcal{L}_A$ obtained in Theorem \ref{Ruelle}. Then
\[
\log(\lambda_A) = \sup_{\mu \in \mathcal{M}_{\sigma}(\mathbb{R}^{\mathbb{N}})}\left\{ h(\mu) + \int_{\mathbb{R}^{\mathbb{N}}} A d\mu \right\} \,.
\]

Moreover, the supremum in the above expression is attained at $\mu_A$.
\end{theorem}

\begin{proof}
Since $A \in \mathcal{H}^+_{\alpha}(\mathbb{R}^{\mathbb{N}})$, follows that
\begin{align}
&\sup_{\mu \in \mathcal{M}_{\sigma}(\mathbb{R}^{\mathbb{N}})}\left\{h(\mu) + \int_{\mathbb{R}^{\mathbb{N}}}A d\mu \right\} \nonumber \\
&\leq \sup_{\mu \in \mathcal{M}_{\sigma}(\mathbb{R}^{\mathbb{N}})}\left\{-\int_{\mathbb{R}^{\mathbb{N}}}A d\mu + \log(\lambda_A) + \int_{\mathbb{R}^{\mathbb{N}}}A d\mu \right\} \nonumber \\
&= \log(\lambda_A) \nonumber \,.
\end{align}

Conversely, we have that
\[
\log(\lambda_A) = h(\mu_A) + \int_{\mathbb{R}^{\mathbb{N}}}A d\mu_A \leq \sup_{\mu \in \mathcal{M}_{\sigma}(\mathbb{R}^{\mathbb{N}})}\left\{ h(\mu) + \int_{\mathbb{R}^{\mathbb{N}}} A d\mu \right\} \,.
\]
\end{proof}

\medskip
\medskip

The following lemma provides conditions to guarantee existence of maximizing probability measures through the existence of ground states.

\begin{lemma}
Let $A \in \mathcal{H}_{\alpha}(\mathbb{R}^{\mathbb{N}})$. If the family $(\mu_{\beta A})_{\beta>0}$ has an accumulation point $\mu_{\infty}$ at infinity, then this point is an $A$-maximizing probability measure.
\label{maximizing-measure}
\end{lemma}

\begin{proof}
By hypothesis, there is a sequence $(\beta_n)_{n \in \mathbb{N}}$, with $\beta_n \to +\infty$, such that $\lim_{n \in \mathbb{N}}\mu_{\beta_n A} = \mu_{\infty}$. Using that $h(\mu) \leq 0$, follows that
\begin{align}
m(A) = \lim_{n \in \mathbb{N}}\frac{1}{\beta_n}\log(\lambda_{\beta_n A})
&= \lim_{n \in \mathbb{N}}\left(\frac{1}{\beta_n}h(\mu_{\beta_n A}) + \int_{\mathbb{R}^{\mathbb{N}}} Ad\mu_{\beta_n A}\right) \nonumber \\
&\leq \int_{\mathbb{R}^{\mathbb{N}}} Ad\mu_{\beta_n A} \nonumber \\
&\leq \int_{\mathbb{R}^{\mathbb{N}}} Ad\mu_{\infty} \nonumber \\
&=m(A) \nonumber \,.
\end{align}

Therefore, $\int_{\mathbb{R}^{\mathbb{N}}} Ad\mu_{\infty} = m(A)$.
\end{proof}

The next proposition shows the existence of ground states. Some interesting results in this direction can be found in \cite{MR3227149, MR1958608, MR3864383, MR2151222, MR2800665, MR2176962}.

\begin{proposition}
Consider $A \in \mathcal{H}_{\alpha}(\mathbb{R}^{\mathbb{N}})$. If there exists $z_0 \in \mathbb{R}$ such that the extension $A' \in \mathcal{H}_{\alpha}(\overline{\mathbb{R}}^{\mathbb{N}})$ satisfies (\ref{decay}) for all $x \in \overline{\mathbb{R}}^{\mathbb{N}}$ and all $n \in \mathbb{N}$, then the family $(\mu_{\beta A})_{\beta>0}$ has an accumulation point $\mu_{\infty}$ at infinity.
\end{proposition}

\begin{proof}
Let $\beta > 0$, and $\beta A' : \overline{\mathbb{R}}^{\mathbb{N}} \to \mathbb{R}$ the $\alpha$-H\"older continuous extension of $\beta A$, by Theorem \ref{Ruelle} we have $\mu_{\beta A'}(\overline{\mathbb{R}}^{\mathbb{N}} \setminus \mathbb{R}^{\mathbb{N}}) = 0$ and $\mu_{\beta A} = \mu_{\beta A'}|_{\mathbb{R}^{\mathbb{N}}}$. Therefore, follows of compactness of $\overline{\mathbb{R}}^{\mathbb{N}}$, that $(\mu_{\beta A'})_{\beta > 0}$ has an accumulation point $\mu'_{\infty}$ at infinity, in other words, there is a sequence $(\beta_n)_{n \in \mathbb{N}}$ with $\beta_n \to \infty$ such that $\lim_{n \in \mathbb{N}}\mu_{\beta_n A'} = \mu'_{\infty}$. By Theorem 5 in \cite{MR2496111} we have that $\mu'_{\infty}$ is an $A'$-maximizing measure, then using part $a)$ of Theorem \ref{calibrated-sub-action} we obtain that this probability measure is supported in $\mathbb{R}^{\mathbb{N}}$.

Let $g : \mathbb{R}^{\mathbb{N}} \to \mathbb{R}$ a Lipschitz continuous function, observe that this implies that $g$ is bounded. Defining $g'(x) = \lim_{y \to x}g(y)$ for each $x \in \overline{\mathbb{R}}^{\mathbb{N}}$, we obtain that $g' : \overline{\mathbb{R}}^{\mathbb{N}} \to \mathbb{R}$ is a bounded Lipschitz continuous function, which implies
\[
\lim_{n \in \mathbb{N}}\int_{\mathbb{R}^{\mathbb{N}}}gd\mu_{\beta_n A} = \lim_{n \in \mathbb{N}}\int_{\overline{\mathbb{R}}^{\mathbb{N}}}g'd\mu_{\beta_n A} = \int_{\overline{\mathbb{R}}^{\mathbb{N}}}g'd\mu'_{\infty} = \int_{\mathbb{R}^{\mathbb{N}}}gd\mu_{\infty} \,.
\]

In other words, $(\mu_{\beta A})_{\beta>0}$ has an accumulation point $\mu_{\infty}$ at infinity, and by Lemma \ref{maximizing-measure} this probability measure is $A$-maximizing.
\end{proof}

\section{Involution kernel}
\label{Involution-kernel-section}

It is widely known that Livsic's Theorem guarantees that a potential with some regularity defined from $\mathbb{R}^{\mathbb{Z}}$ into $\mathbb{R}$ is co-homologous (via the bilateral shift)  to a potential defined from $\mathbb{R}^{\mathbb{N}}$ into $\mathbb{R}$. Conversely, also there exists a tool, known as involution kernel, that provides the desired co-homology between a potential defined from $\mathbb{R}^{\mathbb{N}}$ into $\mathbb{R}$ and a potential defined from $\mathbb{R}^{\mathbb{Z}}$ into $\mathbb{R}$.
The case of the shift acting on $\{1,2,...,d\}^\mathbb{N}$ was considered in
\cite{MR2210682} .

In this section we construct an involution kernel for the non-compact case studied in section \ref{Ruelle-operator-section} and we show some properties that provides an extension of the Gibbs states defined in the last section for the bilateral case joint with an interesting example for the case of stationary Markov probability measures.

We define $(\mathbb{R}^{\mathbb{N}})^* = \{(\ldots, y_2, y_1) \in \mathbb{R}^{\mathbb{N}}\}$ and the map $\sigma^* : (\mathbb{R}^{\mathbb{N}})^* \to (\mathbb{R}^{\mathbb{N}})^*$ as $\sigma^*(\ldots, y_2, y_1) = (\ldots, y_3, y_2)$. For each pair of points $(y, x) \in (\mathbb{R}^{\mathbb{N}})^* \times \mathbb{R}^{\mathbb{N}}$ we will denote by $(y|x)$ the element $(\ldots, y_2, y_1| x_1, x_2, \ldots)$. The set of ordered pairs $(y|x)$ with $x \in \mathbb{R}^{\mathbb{N}}$ and $y \in (\mathbb{R}^{\mathbb{N}})^*$ will be called $\widehat{\mathbb{R}^{\mathbb{N}}}$, which is isomorphic to $\mathbb{R}^{\mathbb{Z}}$. In this case we can define a bilateral sub-shift $\widehat{\sigma} : \widehat{\mathbb{R}^{\mathbb{N}}} \to \widehat{\mathbb{R}^{\mathbb{N}}}$ as the map $\widehat{\sigma}(y|x) = (\tau^*_x(y)|\sigma(x))$, with $\tau^*_x(y) = (\ldots, y_2, y_1, x_1)$.

Now, fixing $A \in \mathcal{H}_{\alpha}(\mathbb{R}^{\mathbb{N}})$, we say that $W: \widehat{\mathbb{R}^{\mathbb{N}}} \to \mathbb{R}$ is an involution kernel for $A$, if the adjunct potential $A^*$ defined by
\[
A^* = A \circ \widehat{\sigma}^{-1} + W \circ \widehat{\sigma}^{-1} - W \,,
\]

depends only of the variable $y$. In \cite{MR2210682} was shown that in fact $A$ has an $\alpha$-H\"older continuous involution kernel defined by
\begin{equation}
W(y|x) = \sum_{n \in \mathbb{N}} A(\tau_{y, n}(x)) - A(\tau_{y, n}(x'))\,,
\label{involution-kernel}
\end{equation}

with $\tau_{y, n}(x) = (y_n ,\ldots, y_1, x_1, x_2, \ldots )$. Note that this involution kernel is $\alpha$-H\"older, because $A \in \mathcal{H}_{\alpha}(\mathbb{R}^{\mathbb{N}})$ and $\tau_{y, n}(x)$ is a contraction. Furthermore, the above implies that $A^*$ is also an $\alpha$-H\"older potential.

Now, we are going to define a natural extension of the Gibbs state associated to $A$ in the bilateral sub-shift $\widehat{\mathbb{R}^{\mathbb{N}}}$. Taking a constant $c \in \mathbb{R}$ satisfying
\[
\int_{\widehat{\mathbb{R}^{\mathbb{N}}}}e^{W(y|x)}d(\rho_{A^*} \times \rho_A)(y, x) = e^c \,,
\]

which is possible because $e^{W(y|x)}$ is an strictly positive function, we define $K(y|x) = e^{W(y|x) - c}$ and, using the above function, $\widehat{\mu}_A$ is defined by
\[
d\widehat{\mu}_A(y, x) = K(y|x) d(\rho_{A^*} \times \rho_A)(y, x) \,.
\]

Using a similar procedure to the one which was  used in \cite{MR2210682} it is possible to show that $\widehat{\mu} \in \mathcal{M}_{\widehat{\sigma}}(\widehat{\mathbb{R}^{\mathbb{N}}})$ and extends the Gibbs states $\mu_A$ and $\mu_{A^*}$ in the following way:

if $\varphi: \widehat{\mathbb{R}^{\mathbb{N}}} \to \mathbb{R}$ is an $\widehat{\mu}_A$-integrable function, such that, $\varphi(y|x) = \varphi(z|x)$, for all $y, z \in (\mathbb{R}^{\mathbb{N}})^*$, then using the notation $\varphi(y|x) = \varphi(x)$, we have
\[
\int_{\mathbb{R}^{\mathbb{N}}}\varphi(x)d\mu_A(x) = \int_{\widehat{\mathbb{R}^{\mathbb{N}}}}\varphi(x)d\widehat{\mu}_A(y, x) \,.
\]

Analogously, if $\varphi: \widehat{\mathbb{R}^{\mathbb{N}}} \to \mathbb{R}$ satisfies $\varphi(y|x) = \varphi(y|z)$ for all $x, z \in \mathbb{R}^{\mathbb{N}}$, then using the notation $\varphi(y|x) = \varphi(y)$, we obtain that
\[
\int_{(\mathbb{R}^{\mathbb{N}})^*}\varphi(y)d\mu_{A^*}(y) = \int_{\widehat{\mathbb{R}^{\mathbb{N}}}}\varphi(y)d\widehat{\mu}_A(y, x) \,.
\]

The proofs of the above claims  are a consequence of ones for  the compact case that appears in \cite{MR3377291}. Here we use the fact that in our case the Gibbs states $\mu_{A'}$ and $\mu_{(A')^*}$ are supported in $\mathbb{R}^{\mathbb{N}}$. Moreover, using the kernel $K$, we can find an explicit form for the eigenfunctions of $\mathcal{L}_A$ and $\mathcal{L}_{A^*}$ associated to $\lambda_A = \lambda_{A^*}$, which are given by $\psi_A(x) = \int_{(\mathbb{R}^{\mathbb{N}})^*} K(y|x) d\rho_{A^*}(y)$ and $\psi_{A^*}(y) = \int_{\mathbb{R}^{\mathbb{N}}} K(y|x) d\rho_A(x)$ respectively.

Besides that, the $j$-th partial derivative of this involution kernel is well defined for all $j \in \mathbb{N}$, when $A$ satisfies the following conditions:
\begin{enumerate}[i)]
\item There exists the partial derivative of $A$ regarding the $j$-th coordinate at the point $x$, for all $x \in \mathbb{R}^{\mathbb{N}}$.
\item Given $\epsilon > 0$, there exists $H_{\epsilon} > 0$, such that, for all $x \in \mathbb{R}^{\mathbb{N}}$, if $h < H_{\epsilon}$, then for all $j \in \mathbb{N}$ it is satisfied the expression
\[
\left| \frac{A(x + he_j) - A(x)}{h} - D_j A(x)\right| < \frac{\epsilon}{2^j} \,.
\]

\end{enumerate}

From the definition of partial derivative and $ii)$, it is possible to demonstrate, in a similar way as in \cite{MR3377291}, that the $j$-th partial derivative of the involution kernel $W$ satisfies the following equation
\begin{equation}
D_j W(y|x) = \sum_{n \in \mathbb{N}} D_{n+j}A(\tau_{y,n}(x)) \,.
\label{derivative-involution-kernel}
\end{equation}

Therefore, using the explicit form of the eigenfunctions (associated to $\mathcal{L}_A$ and $\mathcal{L}_{A^*}$)  and also using the fact that $K(\cdot|x)$ is integrable (regarding $\rho_{A^*}$) and $K(y|\cdot)$ is integrable (with respect to $\rho_A$), we obtain the following expressions for the partial derivatives of the eigenfunctions associated to $\mathcal{L}_A$ and $\mathcal{L}_{A^*}$.
\[
D_j \psi_A(x) = \int_{(\mathbb{R}^{\mathbb{N}})^*} K(y|x) \sum_{n \in \mathbb{N}} D_{n+j}A(\tau_{y,n}(x))d\rho_{A^*}(y) \,,
\]

and
\[
D_j \psi_{A^*}(x) = \int_{\mathbb{R}^{\mathbb{N}}} K(y|x) \sum_{n \in \mathbb{N}} D_{n+j}A(\tau_{y,n}(x))d\rho_A(x) \,.
\]

Henceforth, we will interested in describing the differentiable property  of the eigenfunction of the Ruelle operator for a differentiable potential $A$ (in the case of Markov chains). Some results are known  in the compact setting for the $XY$ model (see \cite{MR3377291}).  We are interested in study partial derivatives of the involution kernel and the entropy for induced stationary Markov measures associated to potentials that depends only of two coordinates, in other words, potentials $A : \mathbb{R}^{\mathbb{N}} \to \mathbb{R}$ such that $A(x) = A(x_1, x_2)$, which implies that in this case when can consider that $A : \mathbb{R}^2 \to \mathbb{R}$.

Stationary Markov measures are defined using a transition kernel $P$, and an stationary measure for $P$, which we will denote by $\theta$. The transition kernel is a strictly positive function $P : \mathbb{R}^2 \to \mathbb{R}$ satisfying
\begin{equation}
\int_{\mathbb{R}}P(x_1, x_2)f(x_2)dx_2 = 1 \,.
\label{transition-kernel}
\end{equation}

The stationary measure for $P$ is an strictly positive function $\theta : \mathbb{R} \to \mathbb{R}$ that satisfies
\begin{equation}
\int_{\mathbb{R}}\theta(x_1)P(x_1, x_2)f(x_1)dx_1 = \theta(x_2) \,.
\label{stationary-measure}
\end{equation}

Using the foregoing equations, we can define the stationary Markov measure induced by $P$ and $\theta$, as the probability measure
\[
\mu([A_1 \ldots A_n]) = \int_{A_1 \times \ldots \times A_n}\theta(x_1)P(x_1, x_2) \ldots P(x_{n-1}, x_n)f(x_n) \ldots f(x_1)dx_1 \ldots dx_n \,.
\]

Besides that, it is possible to show, using (\ref{involution-kernel}), that any $\alpha$-H\"older potential $A : \mathbb{R}^{\mathbb{N}} \to \mathbb{R}$ that depends of two coordinates, has an $\alpha$-H\"older involution kernel defined by $W(y|x) = A(y_1, x_1)$, and in this case $A^*(y) = A(y_2, y_1)$.

The following Theorem shows that Gibbs states associated to potentials that depends of two coordinates are stationary Markov measures and conversely.

\begin{theorem}
\begin{enumerate}[a)]
\item Consider $A \in \mathcal{H}_{\alpha}(\mathbb{R}^{\mathbb{N}})$ such that $A(x) = A(x_1, x_2)$, then there exists an stationary Markov measure $\mu$, that is Gibbs state associated to $A$.
\item Given an stationary Markov measure $\mu$ induced by $P$ and $\theta$, there exists $A \in \mathcal{H}_{\alpha}(\mathbb{R}^{\mathbb{N}})$, with $A(x) = A(x_1, x_2)$, such that this measure is a Gibbs state for $A$.
\end{enumerate}
\end{theorem}

\begin{proof}
\begin{enumerate}[a)]
\item Let $\psi_A$ and $\overline{\psi}_A$ be the eigenfunctions associated to $\mathcal{L}_A$ and $\mathcal{L}_{A^*}$, respectively. Note that $\psi_A(x) = \psi_A(x_1)$ and $\overline{\psi}_A(x) = \overline{\psi}_A(x_1)$, because the potential $A$ depends only on its first two coordinates.

Now we are going to demonstrate that $P_A(x_1, x_2) = \frac{e^{A(x_1, x_2)\overline{\psi}_A(x_2)}}{\lambda_A\overline{\psi}_A(x_1)}$ is a transition kernel, and $\theta(x_1) = \frac{\psi_A(x_1)\overline{\psi}_A(x_1)}{\pi_A}$ is its respectively stationary measure, with $\pi_A = \int_{\mathbb{R}} \psi_A(x_1)\overline{\psi}_A(x_1)f(x_1)dx_1$.

Using a similar procedure to the one which  was used in Theorem 16 of \cite{MR2864625} it is possible to demonstrate that for any $g_n \in \mathcal{H}_{\alpha}(\mathbb{R}^{\mathbb{N}})$, satisfying $g_n(x) = g_n(x_1, \ldots, x_n)$ for some $n \in \mathbb{N}$, it is true  the equality
\[
\int_{\mathbb{R}^{\mathbb{N}}}\mathcal{L}_{\overline{A}} (g_n) d\mu = \int_{\mathbb{R}^{\mathbb{N}}} g_n d\mu \,.
\]

The foregoing equation guarantees part $a)$ of this Theorem, because in the general case, i.e. when $g \in \mathcal{H}_{\alpha}(\mathbb{R}^{\mathbb{N}})$ depends on an arbitrary number of coordinates, defining $g_n(x) = g(x^n)$, with $x^n = (x_1, \ldots, x_n, 1^{\infty})$ we have
\[
|g(x) - g_n(x)| \leq K\widehat{d}(x, x^n)^{\alpha} \leq \frac{K}{2^{n\alpha}}\,.
\]

This implies that $(g_n)_{n \in \mathbb{N}}$ converges pointwise to $g$. Then, by Dominated Convergence Theorem it follows that
\[
\int_{\mathbb{R}^{\mathbb{N}}}\mathcal{L}_{\overline{A}} g d\mu = \lim_{n\in\mathbb{N}}\int_{\mathbb{R}^{\mathbb{N}}}\mathcal{L}_{\overline{A}} g_n d\mu = \lim_{n\in\mathbb{N}}\int_{\mathbb{R}^{\mathbb{N}}} g_n d\mu = \int_{\mathbb{R}^{\mathbb{N}}} g d\mu \,.
\]

\item If $P$ and $\theta$ satisfy (\ref{transition-kernel}) and (\ref{stationary-measure}), respectively, taking $A = \log(P)$, we obtain that
\[
\mathcal{L}_{A^*}(1)(x_1) = \int_{\mathbb{R}}e^{A(x_1, x_2)}f(x_2)dx_2 = \int_{\mathbb{R}}P(x_1, x_2)f(x_2)dx_2 = 1 \,,
\]

which implies that $\lambda_A = 1$ and $\overline{\psi}_A \equiv 1$. Therefore, defining $\theta_A = \frac{\psi_A}{\pi_A}$, it follows that
\begin{align}
\int_{\mathbb{R}}\theta_A(x_1)P(x_1, x_2)f(x_1)dx_1
&= \int_{\mathbb{R}}\frac{\psi_A(x_1)}{\pi_A}e^{A(x_1, x_2)}f(x_1)dx_1 \nonumber \\
&= \frac{1}{\pi_A}\mathcal{L}_A (\psi_A)(x_2) \nonumber \\
&= \frac{\psi_A(x_2)}{\pi_A} \nonumber \\
&= \theta_A(x_2) \,. \nonumber
\end{align}

This implies that the measure induced by $P$ and $\theta$ is a Gibbs state for $A$.
\end{enumerate}
\end{proof}

\begin{proposition}
Let $\mu$ be an stationary Markov measure defined by $P$ and $\theta$, then the entropy of this measure, when  given by
\[
S(\theta P) = -\int_{\mathbb{R}^2}\theta(x_1)P(x_1, x_2)\log(P(x_1, x_2))f(x_2)f(x_1)dx_1dx_2 \,,
\]

coincides with the entropy given by the usual definition for Gibbs states.
\end{proposition}

\begin{proof}
Set $\mu$ an stationary Markov measure induced by $P$ and $\theta$, then, it  follows of part $b)$ of the claim of above Theorem that $\mu$ is a Gibbs measure associated to the normalized potential $A = \log(P)$, therefore the entropy of this measure is given by
\begin{align}
h(\mu)
&= -\int_{\mathbb{R}^2}\log(P(x_1, x_2))d\mu(x_1, x_2) \nonumber \\
&= -\int_{\mathbb{R}^2}\log(P(x_1, x_2))\theta(x_1)P(x_1, x_2)f(x_1)f(x_2)dx_2dx_1 \,. \nonumber
\end{align}
\end{proof}

Now we will make some observations about differentiability of eigenfunction of the Ruelle operator associated to a differentiable potential $A$ that depends of two coordinates. 

Observe that defining $G(x_1) = \int_{\mathbb{R}} e^{A(y_1, x_1)}\varphi(y_1)f(y_1)dy_1$, with $\varphi: \mathbb{R} \to \mathbb{R}$ an $\alpha$-H\"older continuous function, and using the fact that $e^{A(y_1, \cdot)}\varphi(y_1)f(y_1)$ is integrable for all $x_1 \in \mathbb{R}$ and differentiable for each $y_1 \in \mathbb{R}$, it follows from (\ref{derivative-involution-kernel}) that
\[
\frac{\partial G}{\partial x_1}(x_1) = \int_{\mathbb{R}}e^{A(y_1, x_1)}\frac{\partial A}{\partial x_1}(y_1, x_1)\varphi(y_1)f(y_1)dy_1 \,.
\]

Furthermore, in particular when $\varphi = \psi_A$, we obtain
\begin{align}
\frac{\partial \psi_A}{\partial x_1}(x_1)
&= \frac{1}{\lambda_A}\int_{\mathbb{R}}e^{A(y_1, x_1)}\frac{\partial A}{\partial x_1}(y_1, x_1)\psi_A(y_1)f(y_1)dy_1 \nonumber \\
&= \frac{1}{\lambda_A}\int_{\mathbb{R}}e^{A(y_1, x_1)}\frac{\partial A}{\partial x_1}(y_1, x_1)d\rho_{A^*}(y_1) \nonumber \,.
\end{align}

Which coincides with the derivative obtained in the general case using the involution kernel approach.

\section{\emph{FKG}-inequalities}
\label{FKG-inequalities-section}

\emph{DLR}-Gibbs probability measures are an interesting topic in Statistical Mechanics. In this section we are going to show existence of this type of measures in a similar fashion as  in  sections \ref{Ruelle-operator-section} and \ref{Involution-kernel-section}. Moreover, we are going to show that these probability measures satisfy a \emph{FKG}-inequality and we will show the connection between \emph{DLR}-Gibbs probability measures, Thermodynamic limit Gibbs probability measures and Gibbs states such as was defined in section \ref{Ruelle-operator-section}.

If $x, y \in \mathbb{R}^{\mathbb{N}}$, we say that $x \preceq y$ if, and only if $x_i \leq y_i$ for each $i \in \mathbb{N}$. Using the above definition we say that a function $\varphi : \mathbb{R}^{\mathbb{N}} \to \mathbb{R}$ is an increasing function, if for any pair $x, y \in \mathbb{R}^{\mathbb{N}}$ such that $x \preceq y$, we have $\varphi(x) \leq \varphi(y)$.

For each $n \in \mathbb{N}$, $t \in \mathbb{R}$, and $x, y, z \in \mathbb{R}^{\mathbb{N}}$ we are going to use the following notation
\begin{align}
[x|y]_n &= (x_1, \ldots, x_n, y_{n+1}, y_{n+2}, \ldots) \in \mathbb{R}^{\mathbb{N}} \nonumber \,, \\
[x|t|y]_n &= (x_1, \ldots, x_n, t, y_{n+2}, y_{n+3}, \ldots) \in \mathbb{R}^{\mathbb{N}} \nonumber \,, \\
[x|y|z]_{n, n+r} &= (x_1, \ldots, x_n, y_{n+1}, \ldots, y_{n+r}, z_{n+r+1}, z_{n+r+2}, \ldots) \in \mathbb{R}^{\mathbb{N}} \nonumber \,, \\
[t|y]_1 &= (t, y_2, y_3, \ldots) \in \mathbb{R}^{\mathbb{N}} \nonumber \,.
\end{align}

Note that under the above notation $[t|\sigma^n(y)]_1 = (t, y_{n+2}, y_{n+3}, \ldots) \in \mathbb{R}^{\mathbb{N}}$.

Fixing $A \in \mathcal{H}_{\alpha}(\mathbb{R}^{\mathbb{N}})$, $y \in \mathbb{R}^{\mathbb{N}}$, and $n \in \mathbb{N}$, we define the probability measure $\mu^y_n$ on all the Borelian sets in $\mathbb{R}^{\mathbb{N}}$, assigning to each $E \subset \mathbb{R}^{\mathbb{N}}$ the value
\begin{equation}
\mu^y_n(E) = \frac{1}{Z^y_n}\int_{\mathbb{R}^n}e^{S_n A([x|y]_n)}\chi_E([x|y]_n)f(x_1) \ldots f(x_n)dx_1 \ldots dx_n \,,
\label{probability-finite-set}
\end{equation}

with $Z^y_n = \int_{\mathbb{R}^n}e^{S_n A([x|y]_n)}f(x_1) \ldots f(x_n)dx_1 \ldots dx_n$. In this case, the integral of any $\alpha$-H\"older continuous function $\varphi : \mathbb{R}^{\mathbb{N}} \to \mathbb{R}$ is defined by
\begin{align}
\int_{\mathbb{R}^{\mathbb{N}}} \varphi d\mu^y_n
&= \int_{\mathbb{R}^{\mathbb{N}}} \varphi([x|y]_n) d\mu^{[x|y]_n}_n(x) \nonumber \\
&= \frac{1}{Z^y_n}\int_{\mathbb{R}^n}e^{S_n A([x|y]_n)}\varphi([x|y]_n)f(x_1) \ldots f(x_n)dx_1 \ldots dx_n \,.
\end{align}

In the next Lemma we are going to show that the family of probability measures defined in (\ref{probability-finite-set}) with $n \in\mathbb{N}$ and $y \in\ \mathbb{R}^{\mathbb{N}}$, is in fact a Gibbsian specification, with kernel $K_n(E, y) = \mu^y_n(E)$.

\begin{lemma}
The family $K_n : (\mathcal{B}, \mathbb{R}^{\mathbb{N}}) \to [0, 1]$ with $n \in \mathbb{N}$, defined by
\[
K_n(E, y) = \mu^y_n(E) \,,
\]

is a Gibbsian specification, in other words, $(K_n)_{n \in \mathbb{N}}$ satisfies the following properties
\begin{enumerate}[a)]
\item The map $y \mapsto K_n(E, y)$ is $\sigma^n\mathcal{B}$-measurable for any $E \in \mathcal{B}$.
\item The map $E \mapsto K_n(E, y)$ is a probability measure for each $y \in \mathbb{R}^{\mathbb{N}}$.
\item For any $n \in \mathbb{N}$, $r \in \mathbb{N} \cup \{0\}$, and any $\alpha$-H\"older continuous function $\varphi : \mathbb{R}^{\mathbb{N}} \to \mathbb{R}$ we have the compatibility condition
\[
K_{n+r}(\varphi, z) = K_{n+r}(K_n(\varphi, y), z) \,,
\]
with $K_n(\varphi, y) = \int_{\mathbb{R}^{\mathbb{N}}}\varphi d\mu^y_n$.
\end{enumerate}
\label{Gibbsian-specification}
\end{lemma}

\begin{proof}
Note that for each $n \in \mathbb{N}$ and $y \in \mathbb{R}^{\mathbb{N}}$ we have
\[
K_n(E, y) = \frac{\mathcal{L}^n_A(\chi_E)(\sigma^n(y))}{\mathcal{L}^n_A(1)(\sigma^n(y))} \,,
\]
which implies that the map $y \mapsto K_n(E, y)$ is $\sigma^n\mathcal{B}$-measurable for any $E \in \mathcal{B}$, therefore part $a)$ of this Lemma is obtained. Part $b)$ is obvious by the definition of $K_n$. In order to demonstrate part $c)$ of this Lemma, we observe that the equation $K_{n+r}(\varphi, z) = K_{n+r}(K_n(\varphi, y), z)$ is a direct consequence of the following equality
\[
\int_{\mathbb{R}^{\mathbb{N}}} \varphi([y|z]_{n+r}) d\mu^{[y|z]_{n+r}}_{n+r}(y) = \int_{\mathbb{R}^{\mathbb{N}}} \left(\int_{\mathbb{R}^{\mathbb{N}}} \varphi([x|y|z]_{n, n+r})d\mu^{[x|y|z]_{n, n+r}}_n(x)\right) d\mu^{[y|z]_{n+r}}_{n+r}(y) \,.
\]

Therefore, defining
\[
\psi([y|z]_{n+r}) = \int_{\mathbb{R}^{\mathbb{N}}} \varphi([x|y|z]_{n, n+r})d\mu^{[x|y|z]_{n, n+r}}_n(x) \,,
\]

it will be enough to show that
\[
\int_{\mathbb{R}^{\mathbb{N}}} \varphi([y|z]_{n+r}) d\mu^{[y|z]_{n+r}}_{n+r}(y) = \int_{\mathbb{R}^{\mathbb{N}}} \psi([y|z]_{n+r}) d\mu^{[y|z]_{n+r}}_{n+r}(y)\,,
\]

which is equivalent to show that
\begin{align}
&\int_{\mathbb{R}^{n+r}} e^{S_{n+r}A([y|z]_{n+r})}\varphi([y|z]_{n+r}) f(y_1) \ldots f(y_{n+r}) dy_1 \ldots dy_{n+r} \nonumber \\
&= \int_{\mathbb{R}^{n+r}} e^{S_{n+r}A([y|z]_{n+r})}\psi([y|z]_{n+r}) f(y_1) \ldots f(y_{n+r}) dy_1 \ldots dy_{n+r} \nonumber \,.
\end{align}

Indeed, observe that
\begin{align}
&\int_{\mathbb{R}^{n+r}} e^{S_{n+r}A([y|z]_{n+r})}\psi([y|z]_{n+r}) f(y_1) \ldots f(y_{n+r}) dy_1 \ldots dy_{n+r} \nonumber \\
&= \int_{\mathbb{R}^{n+r}} \frac{1}{Z^{[y|z]_{n+r}}} \nonumber \\
&\ \ \ \ \left(\int_{\mathbb{R}^n} e^{ S_{n+r}A([y|z]_{n+r}) + S_n A([x|y|z]_{n,n+r})} \varphi([x|y|z]_{n,n+r}) f(x_1) \ldots f(x_n) dx_1 \ldots dx_n\right) \nonumber \\
&\ \ \ \ f(y_1) \ldots f(y_{n+r}) dy_1 \ldots dy_{n+r} \nonumber \,.
\end{align}
Now, using the fact that any ergodic sum satisfies
\[
S_{n+r}A([y|z]_{n+r}) + S_n A([x|y|z]_{n,n+r}) = S_{n+r}A([x|y|z]_{n,n+r}) + S_n A([y|z]_{n+r}) \,,
\]
it follows that the above integral is equal to
\begin{align}
&\int_{\mathbb{R}^{n+r}} \frac{1}{Z^{[y|z]_{n+r}}} \nonumber \\
&\ \ \ \ \left(\int_{\mathbb{R}^n} e^{ S_{n+r}A([x|y|z]_{n,n+r}) + S_n A([y|z]_{n+r})} \varphi([x|y|z]_{n,n+r}) f(x_1) \ldots f(x_n) dx_1 \ldots dx_n\right) \nonumber \\
&\ \ \ \ f(y_1) \ldots f(y_{n+r}) dy_1 \ldots dy_{n+r} \nonumber \\
&= \int_{\mathbb{R}^r} \bigl(\int_{\mathbb{R}^n} \frac{1}{Z^{[y|z]_{n+r}}} \nonumber \\
&\ \ \ \ \left(\int_{\mathbb{R}^n} e^{ S_{n+r}A([x|y|z]_{n,n+r}) + S_n A([y|z]_{n+r})} \varphi([x|y|z]_{n,n+r}) f(x_1) \ldots f(x_n) dx_1 \ldots dx_n\right) \nonumber \\
&\ \ \ \ f(y_1) \ldots f(y_n) dy_1 \ldots dy_n\bigr) f(y_{n+1}) \ldots f(y_{n+r}) dy_{n+1} \ldots dy_{n+r} \nonumber \\
&= \int_{\mathbb{R}^r} \left(\int_{\mathbb{R}^n} \frac{e^{S_n A([y|z]_{n+r})}}{Z^{[y|z]_{n+r}}} f(y_1) \ldots f(y_n) dy_1 \ldots dy_n\right) \nonumber \\
&\ \ \ \ \left(\int_{\mathbb{R}^n} e^{ S_{n+r}A([x|y|z]_{n,n+r})} \varphi([x|y|z]_{n,n+r}) f(x_1) \ldots f(x_n) dx_1 \ldots dx_n\right) \nonumber \\
&\ \ \ \ f(y_{n+1}) \ldots f(y_{n+r}) dy_{n+1} \ldots dy_{n+r} \nonumber 
\end{align}

\begin{align}
&= \int_{\mathbb{R}^r} \left(\int_{\mathbb{R}^n} e^{ S_{n+r}A([x|y|z]_{n,n+r})} \varphi([x|y|z]_{n,n+r}) f(x_1) \ldots f(x_n) dx_1 \ldots dx_n\right) \nonumber \\
&\ \ \ \ f(y_{n+1}) \ldots f(y_{n+r}) dy_{n+1} \ldots dy_{n+r} \nonumber \\
&= \int_{\mathbb{R}^{n+r}} e^{ S_{n+r}A([y|z]_{n+r})} \varphi([y|z]_{n+r}) f(y_1) \ldots f(y_{n+r}) dy_1 \ldots dy_{n+r} \nonumber \,.
\end{align}

This concludes our proof.
\end{proof}

Fixing a Gibbsian specification $(K_n)_{n \in \mathbb{N}}$ determined by an $\alpha$-H\"older continuous potential $A$, we say that a probability measure $\mu$ is a \emph{DLR}-Gibbs probability measure associated to $A$, if $E_{\mu}(\varphi|\sigma^n\mathcal{B})(y) = K_n(\varphi, y)$ for almost every point $y \in \mathbb{R}^{\mathbb{N}}$, any $\alpha$-H\"older continuous function $\varphi$ and each $n \in \mathbb{N}$. The set of all such $\mu$ will be denoted from now on by $\mathcal{G}^{DLR}(A)$.

On the other hand, the set of Thermodynamic limit Gibbs probability measures, denoted by $\mathcal{G}^{TL}(A)$, is defined as the closure of the convex hull of the set of cluster points of $\{K_n(\cdot, y_n) : n \in \mathbb{N}, y_n \in \mathbb{R}^{\mathbb{N}}\}$.

The next Lemma shows that $\mathcal{G}^{TL}(A) \subset \mathcal{G}^{DLR}(A)$ using a classical approach in statistical mechanics known as \emph{DLR}-equations.

\begin{lemma}
Consider $(K_n)_{n \in \mathbb{N}}$ be the Gibbsian specification determined by kernels of the form $K_n(\varphi, z) = \int_{\mathbb{R}^{\mathbb{N}}}\varphi d\mu^z_n$. If there exists a sequence of natural numbers $(n_j)_{j \in \mathbb{N}}$, such that $\lim_{j \in\mathbb{N}} K_{n_j}(\cdot, z) = \lim_{j \in\mathbb{N}} \mu^z_{n_j} = \mu^z$ in the weak* topology. Then, for any $\alpha$-H\"older continuous function $\varphi : \mathbb{R}^{\mathbb{N}} \to \mathbb{N}$ we have
\[
\int_{\mathbb{R}^{\mathbb{N}}} \varphi(y) d\mu^z(y) = \int_{\mathbb{R}^{\mathbb{N}}} \left(\int_{\mathbb{R}^{\mathbb{N}}} \varphi([x|y]_n) d\mu^{[x|y]_n}_n(x)\right) d\mu^z(y) \,.
\]
\end{lemma}

\begin{proof}
By the Lemma (\ref{Gibbsian-specification}) we have
\[
\int_{\mathbb{R}^{\mathbb{N}}} \varphi([y|z]_{n+r}) d\mu^{[y|z]_{n+r}}_{n+r}(y) = \int_{\mathbb{R}^{\mathbb{N}}} \psi([y|z]_{n+r}) d\mu^{[y|z]_{n+r}}_{n+r}(y) \,,
\]

with $\psi([y|z]_{n+r}) = \int_{\mathbb{R}^{\mathbb{N}}} \varphi([x|y|z]_{n, n+r})d\mu^{[x|y|z]_{n, n+r}}_n(x)$. Therefore, taking the limit when $r$ goes to infinity in both of the sides of the equation, we obtain that
\[
\int_{\mathbb{R}^{\mathbb{N}}} \varphi(y) d\mu^z(y) = \int_{\mathbb{R}^{\mathbb{N}}} \psi(y) d\mu^z(y) \,.
\]

In other words
\[
\int_{\mathbb{R}^{\mathbb{N}}} \varphi(y) d\mu^z(y) = \int_{\mathbb{R}^{\mathbb{N}}} \left(\int_{\mathbb{R}^{\mathbb{N}}} \varphi([x|y]_n) d\mu^{[x|y]_n}_n(x)\right) d\mu^z(y) \,.
\]
\end{proof}

In the next Theorem we will show the relation between the conformal probability measure and the Gibbs state associated to the Ruelle operator $\mathcal{L}_A$, with the set of \emph{DLR}-Gibbs probability measures associated to $A$.

\begin{theorem}
Suppose that $A \in \mathcal{H}_{\alpha}(\mathbb{R}^{\mathbb{N}})$ and $(K_n)_{n \in \mathbb{N}}$ a Gibbsian specification such as defined in Lemma \ref{Gibbsian-specification}. Then, the eigenmeasure $\rho_A$ and the Gibbs state $\mu_A$, such as were defined in Theorem \ref{Ruelle}, belongs to $\mathcal{G}^{DLR}(A)$.
\end{theorem}

\begin{proof}
In the first case it is enough to show  that $E_{\rho_A}(\varphi | \sigma^n \mathcal{B})(z) = K_n(\varphi, z)$ for almost every point $z \in \mathbb{R}^{\mathbb{N}}$, any $\alpha$-H\"older continuous function $\varphi$ and each $n \in\mathbb{N}$. Indeed,
\begin{align}
\int_{\mathbb{R}^{\mathbb{N}}} K_n(\varphi, z) d\rho_A(z)
&= \int_{\mathbb{R}^{\mathbb{N}}}\frac{\mathcal{L}^n_A (\varphi) (\sigma^n(z))}{\mathcal{L}^n_A (1) (\sigma^n(z))} d\rho_A(z) \nonumber \\
&= \int_{\mathbb{R}^{\mathbb{N}}}\frac{\mathcal{L}^n_A (\varphi) (\sigma^n(z))}{\mathcal{L}^n_A (1) (\sigma^n(z))} \frac{1}{\psi_A(z)} d\mu_A(z) \nonumber \\
&= \int_{\mathbb{R}^{\mathbb{N}}}\frac{1}{\lambda^n_A}\mathcal{L}^n_A \left(\frac{\mathcal{L}^n_A (\varphi) (\sigma^n(y))}{\mathcal{L}^n_A (1) (\sigma^n(y))}\right)(z) \frac{1}{\psi_A(z)} d\mu_A(z) \nonumber \\
&= \int_{\mathbb{R}^{\mathbb{N}}}\frac{1}{\lambda^n_A}\mathcal{L}^n_A \left(\frac{\mathcal{L}^n_A (\varphi) (\sigma^n(y))}{\mathcal{L}^n_A (1) (\sigma^n(y))}\right)(\sigma^n(z)) \frac{1}{\psi_A(\sigma^n(z))} d\mu_A(z) \nonumber \,.
\end{align}

Now, using the fact  that for any $n \in \mathbb{N}$ it is satisfied the equation $K_n(K_n(\varphi, y), z) = K_n(\varphi, z)$ (proved   in Lemma \ref{Gibbsian-specification}), it follows that the above integral is equal to
\begin{align}
&\int_{\mathbb{R}^{\mathbb{N}}}\frac{1}{\lambda^n_A}\mathcal{L}^n_A (\varphi)(\sigma^n(z)) \frac{1}{\psi_A(\sigma^n(z))} d\mu_A(z) \nonumber \\
&= \int_{\mathbb{R}^{\mathbb{N}}}\frac{1}{\lambda^n_A}\mathcal{L}^n_A (\varphi)(z) \frac{1}{\psi_A(z)} d\mu_A(z) \nonumber \\
&= \int_{\mathbb{R}^{\mathbb{N}}}\frac{1}{\lambda^n_A}\mathcal{L}^n_A (\varphi)(z) d\rho_A(z) \nonumber \\
&= \int_{\mathbb{R}^{\mathbb{N}}}\varphi(z) d\rho_A(z) \nonumber \,.
\end{align}

The foregoing implies that $\rho_A \in \mathcal{G}^{DLR}(A)$. The proof for $\mu_A$ follows a procedure similar to the above using the fact that $\mathcal{L}_{\overline{A}}\mu_A = \mu_A$ and also that $\overline{A}$ is an $\alpha$-H\"older continuous potential.
\end{proof}

The main result of this section is to show that probability measures associated to kernels of the Gibbsian specification (such as was defined above) have positive correlation, which is usually knowing as \emph{FKG}-inequality. The following Lemma provide necessary tools to prove this result.

\begin{lemma}
Let $\eta$ a Borel probability measure on $\mathbb{R}$. If $\varphi, \psi : \mathbb{R} \to \mathbb{R}$ are increasing integrable functions, then
\[
\int_{\mathbb{R}}\varphi\psi d\eta \geq \int_{\mathbb{R}} \varphi d\eta \cdot \int_{\mathbb{R}} \psi d\eta \,.
\]
\label{FKG-inequality-dim1}
\end{lemma}

\begin{proof}
The proof is similar to the one in \cite{CiLo17}.
\end{proof}

The following Lemma shows a decomposition of $\mu^y_{n+1}$ in terms of $\mu^y_n$ and a suitable Borel measure defined on the set of real numbers.

\begin{lemma}
Let $y \in \mathbb{R}^{\mathbb{N}}$ and $\varphi \in \mathcal{H}_{\alpha}(\mathbb{R}^{\mathbb{N}})$. Then the following equation is valid for each $n \in \mathbb{N}$.
\[
\int_{\mathbb{R}}\left(\int_{\mathbb{R}^{\mathbb{N}}} \varphi([x|t|y]_n) d\mu^{[x|t|y]_n}_n(x) \right)d\eta(t) = \int_{\mathbb{R}^{\mathbb{N}}} \varphi([x|y]_{n+1}) d\mu^{[x|y]_{n+1}}_{n+1}(x) \,,
\]
with $\eta(E) = \frac{Z^{[y|t|y]_n}_n}{Z^y_{n+1}}\int_{\mathbb{R}} e^{A([t|\sigma^n(y)]_1)} \chi_E(t) f(t) dt$, for each mensurable set $E \subset \mathbb{R}$.
\end{lemma}

\begin{proof}
It follows from the definitions of $\eta$ and $\mu^y_n$ that
\begin{align}
&\int_{\mathbb{R}} \left(\int_{\mathbb{R}^{\mathbb{N}}} \varphi([x|t|y]_n) d\mu^{[x|t|y]_n}_n(x) \right) d\eta(t) \nonumber \\
&= \int_{\mathbb{R}} \frac{Z^{[y|t|y]_n}_n}{Z^y_{n+1}} e^{A([t|\sigma^n(y)]_1)}\left(\int_{\mathbb{R}^{\mathbb{N}}} \varphi([x|t|y]_n) d\mu^{[x|t|y]_n}_n(x) \right) f(t) dt \nonumber \\
&= \frac{1}{Z^y_{n+1}} \int_{\mathbb{R}} e^{A([t|\sigma^n(y)]_1)} \nonumber \\
& \ \ \ \ \left(\int_{\mathbb{R}^n} e^{S_n A([x|t|y]_n)} \varphi([x|t|y]_n) f(x_1) \ldots f(x_n) dx_1 \ldots dx_n \right) f(t) dt \nonumber \\
&= \frac{1}{Z^y_{n+1}} \int_{\mathbb{R}^{n+1}} e^{S_{n+1} A([x|y]_{n+1})} \varphi([x|y]_{n+1}) f(x_1) \ldots f(x_{n+1}) dx_1 \ldots dx_{n+1} \nonumber \\
&= \int_{\mathbb{R}^{\mathbb{N}}} \varphi([x|y]_{n+1}) d\mu^{[x|y]_{n+1}}_{n+1}(x) \nonumber \,.
\end{align}

We want to study the behavior of the map $t \mapsto \int_{\mathbb{R}^{\mathbb{N}}} \varphi d\mu^{[y|t|y]_n}_n$.  This is an important tool to show that the probability measures defined in (\ref{probability-finite-set}) have positive correlation. On this way, it is necessary to define an especial class of functions. We say that an $\alpha$-H\"older continuous potential $A : \mathbb{R}^{\mathbb{N}} \to \mathbb{R}$ belongs to the class $\mathcal{E}$, if it is continuously differentiable in each coordinate, and the derivative of the $n$-th ergodic sum regarding each coordinate defined by the map
\[
(x_1, x_2, \ldots) \mapsto \frac{d}{dt}S_n A([x|t|y]_n) \,,
\]

is increasing. Note that the function defined above depends only of its first $n$ coordinates.

On the other hand, fixing $n \in \mathbb{N}$, $y \in \mathbb{R}^{\mathbb{N}}$, and a potential $A$ in the class $\mathcal{E}$, we say that the probability measure $\mu^y_n$, such as was defined in (\ref{probability-finite-set}), satisfies the \emph{FKG}-inequality if
\begin{equation}
\int_{\mathbb{R}^{\mathbb{N}}} \varphi \psi d\mu^y_n \geq \int_{\mathbb{R}^{\mathbb{N}}} \varphi d\mu^y_n \cdot \int_{\mathbb{R}^{\mathbb{N}}} \psi d\mu^y_n \,.
\label{FKG-inequality}
\end{equation}

for any pair of increasing $\alpha$-H\"older continuous functions $\varphi, \psi$ from $\mathbb{R}^{\mathbb{N}}$ into $\mathbb{R}$ that depends only of its first $n$ coordinates.
\end{proof}

The following Lemma shows that under suitable conditions the map $t \mapsto \int_{\mathbb{R}^{\mathbb{N}}} \varphi d\mu^{[y|t|y]_n}_n$ is increasing.

\begin{lemma}
Let $n \in \mathbb{N}$ and $y \in \mathbb{R}^{\mathbb{N}}$ fixed. If the potential $A : \mathbb{R}^{\mathbb{N}} \to \mathbb{R}$ belongs to the class $\mathcal{E}$ and the probability measure $\mu^y_n$ satisfies (\ref{FKG-inequality}). Then, for any $\alpha$-H\"older continuous increasing function $\varphi : \mathbb{R}^{\mathbb{N}} \to \mathbb{R}$ that depends only of its first $n$ coordinates, the map
\[
t \mapsto \int_{\mathbb{R}^{\mathbb{N}}} \varphi d\mu^{[y|t|y]_n}_n \,,
\]

is a real increasing function.
\label{increasing-integral}
\end{lemma}

\begin{proof}
Since $\varphi$ depends only of its first $n$ coordinates, we have
\[
\int_{\mathbb{R}^{\mathbb{N}}}\varphi([x|t|y]_n)d\mu^{[x|t|y]_n}_n(x) = \int_{\mathbb{R}^{\mathbb{N}}}\varphi([x|y]_n)d\mu^{[x|t|y]_n}_n(x) \,.
\]

On the other hand, $A$ belongs to the class $\mathcal{E}$, then the function $\frac{e^{S_n A([x|t|y]_n)}}{Z^{[y|t|y]_n}_n}$ is differentiable regarding the variable $t$. Therefore, using the fact that
\begin{align}
&\frac{d}{dt}\int_{\mathbb{R}^{\mathbb{N}}}\varphi([x|y]_n)d\mu^{[x|t|y]_n}_n(x) \nonumber \\
&= \frac{d}{dt} \frac{1}{Z^{[y|t|y]_n}_n} \int_{\mathbb{R}^n}\varphi([x|y]_n) e^{S_n A([x|t|y]_n)} f(x_1) \ldots f(x_n) dx_1 \ldots dx_n \nonumber \\
&= \int_{\mathbb{R}^n}\varphi([x|y]_n)\frac{d}{dt}\left(\frac{e^{S_n A([x|t|y]_n)}}{Z^{[y|t|y]_n}_n}\right) f(x_1) \ldots f(x_n) dx_1 \ldots dx_n \,,
\label{derivative-integral}
\end{align}

it is enough to show  that $\frac{d}{dt}\left(\frac{e^{S_n A([x|t|y]_n)}}{Z^{[y|t|y]_n}_n}\right)$ is non-negative.

Note that by the derivative quotient rule for functions from $\mathbb{R}$ into $\mathbb{R}$, we have
\[
\frac{d}{dt}\left(\frac{e^{S_n A([x|t|y]_n)}}{Z^{[y|t|y]_n}_n}\right) = \frac{e^{S_n A([x|t|y]_n)}}{Z^{[y|t|y]_n}_n} \left(\frac{d}{dt}S_n A([x|t|y]_n) - \frac{1}{Z^{[y|t|y]_n}_n}\frac{d}{dt} Z^{[y|t|y]_n}_n \right) \,.
\]

Besides that, the last term in right side of the equation is equal to
\begin{align}
&\frac{1}{Z^{[y|t|y]_n}_n}\frac{d}{dt} Z^{[y|t|y]_n}_n \nonumber \\
&= \frac{1}{Z^{[y|t|y]_n}_n} \int_{\mathbb{R}^n} \frac{d}{dt} \left(e^{S_n A([v|t|y]_n)}\right) f(v_1) \ldots f(v_n) dv_1 \ldots dv_n \nonumber \\
&= \frac{1}{Z^{[y|t|y]_n}_n} \int_{\mathbb{R}^n} e^{S_n A([v|t|y]_n)} \frac{d}{dt} \left(S_n A([v|t|y]_n)\right) f(v_1) \ldots f(v_n) dv_1 \ldots dv_n \nonumber \\
&= \int_{\mathbb{R}^{\mathbb{N}}}\frac{d}{dt}\left(S_n A([v|t|y]_n)\right) d\mu^{[v|t|y]_n}_n(v) \nonumber \,.
\end{align}

Therefore, by the above equality we obtain that
\begin{align}
&\frac{d}{dt}\left(\frac{e^{S_n A([x|t|y]_n)}}{Z^{[y|t|y]_n}_n}\right) \nonumber \\
&= \frac{e^{S_n A([x|t|y]_n)}}{Z^{[y|t|y]_n}_n} \left(\frac{d}{dt}S_n A([x|t|y]_n) - \int_{\mathbb{R}^{\mathbb{N}}}\frac{d}{dt}\left(S_n A([v|t|y]_n)\right) d\mu^{[v|t|y]_n}_n(v) \right) \,.
\label{derivative-integral-interior}
\end{align}

Then, replacing (\ref{derivative-integral-interior}) in (\ref{derivative-integral}), we get the following
\begin{align}
&\frac{d}{dt}\int_{\mathbb{R}^{\mathbb{N}}}\varphi([x|y]_n)d\mu^{[x|t|y]_n}_n(x) \nonumber \\
&= \int_{\mathbb{R}^n} \varphi([x|y]_n) \frac{e^{S_n A([x|t|y]_n)}}{Z^{[y|t|y]_n}_n} \nonumber \\
&\ \ \ \ \left(\frac{d}{dt}S_n A([x|t|y]_n) - \int_{\mathbb{R}^{\mathbb{N}}}\frac{d}{dt}\left(S_n A([v|t|y]_n)\right) d\mu^{[v|t|y]_n}_n(v) \right) f(x_1) \ldots f(x_n) dx_1 \ldots dx_n \nonumber \\
&= \int_{\mathbb{R}^{\mathbb{N}}} \varphi([x|y]_n) \frac{d}{dt}\left(S_n A([x|t|y]_n)\right) d\mu^{[x|t|y]_n}_n(x) \nonumber \\
&\ \ \ \  - \int_{\mathbb{R}^{\mathbb{N}}} \varphi([x|y]_n) d\mu^{[x|t|y]_n}_n(x) \cdot \int_{\mathbb{R}^{\mathbb{N}}}\frac{d}{dt}\left(S_n A([v|t|y]_n)\right) d\mu^{[v|t|y]_n}_n(v) \nonumber \\
&\geq 0 \nonumber \,.
\end{align}

Observe that the last inequality is a consequence of (\ref{FKG-inequality}).
\end{proof}

The following Theorem, which is the main result of this section, provides some conditions in order to guarantee that each probability measure $\mu^y_n$ such as was defined in (\ref{probability-finite-set}) satisfies the \emph{FKG}-inequality. In other words, each one of the members of the Gibbsian specification $(K_n)_{n \in \mathbb{N}}$ satisfies the \emph{FKG}-inequality.

\begin{theorem}
Let $A : \mathbb{R}^{\mathbb{N}} \to \mathbb{R}$ be a potential that belongs to the class $\mathcal{E}$. Then, for any $n \in \mathbb{N}$ and each $y \in \mathbb{R}^{\mathbb{N}}$, the probability measure $\mu^y_n$, which was defined in (\ref{probability-finite-set}) satisfies (\ref{FKG-inequality}).
\end{theorem}

\begin{proof}
We are going to demonstrate this theorem by induction. The case $n = 1$ follows directly of Lemma (\ref{FKG-inequality-dim1}), since the probability measure $\mu^y_1$ is defined on the Borelian sets in $\mathbb{R}$.

We assume the claim of this Theorem for $n \geq 1$, and we will demonstrate that this implies the same result for $n + 1$. Let $\varphi$ and $\psi$ increasing $\alpha$-H\"older continuous functions that depends only of its first $n+1$ coordinates, then it follows from the definition of $\mu^y_{n+1}$ that
\begin{align}
&\int_{\mathbb{R}^{\mathbb{N}}} \varphi([x|y]_{n+1})\psi([x|y]_{n+1}) d\mu^{[x|y]_{n+1}}_{n+1}(x) \nonumber \\
&= \frac{1}{Z^y_{n+1}}\int_{\mathbb{R}^{n+1}}e^{S_{n+1} A([x|y]_{n+1})}\varphi([x|y]_{n+1})\psi([x|y]_{n+1}) f(x_1) \ldots f(x_{n+1})dx_1 \ldots dx_{n+1} \nonumber \\
&= \frac{1}{Z^y_{n+1}}\int_{\mathbb{R}} e^{A([t|\sigma^n(y)]_1)} \nonumber \\
& \ \ \ \ \left(\int_{\mathbb{R}^n}e^{S_n A([x|t|y]_n)}\varphi([x|t|y]_n)\psi([x|t|y]_n) f(x_1) \ldots f(x_n)dx_1 \ldots dx_n\right)f(t)dt \nonumber \\
&= \int_{\mathbb{R}} \frac{Z^{[y|t|y]}_n}{Z^y_{n+1}} e^{A([t|\sigma^n(y)]_1)} \left(\int_{\mathbb{R}^{\mathbb{N}}}\varphi([x|t|y]_n)\psi([x|t|y]_n) d\mu^{[x|t|y]_n}_n(x)\right)f(t)dt \nonumber \\
&= \int_{\mathbb{R}} \left(\int_{\mathbb{R}^{\mathbb{N}}}\varphi([x|t|y]_n)\psi([x|t|y]_n) d\mu^{[x|t|y]_n}_n(x)\right) d\eta(t) \nonumber \\
&\geq \int_{\mathbb{R}} \left(\int_{\mathbb{R}^{\mathbb{N}}}\varphi([x|t|y]_n)d\mu^{[x|t|y]_n}_n(x)\right) \cdot \left(\int_{\mathbb{R}^{\mathbb{N}}}\psi([x|t|y]_n) d\mu^{[x|t|y]_n}_n(x)\right)d\eta(t) \nonumber \,,
\end{align}

where the last inequality is obtained by the induction hypothesis. On other hand, by Lemma (\ref{increasing-integral}) we have that the maps $t \mapsto \int_{\mathbb{R}^{\mathbb{N}}} \varphi d\mu^{[y|t|y]_n}_n$ and $t \mapsto \int_{\mathbb{R}^{\mathbb{N}}} \psi d\mu^{[y|t|y]_n}_n$ are real increasing functions. Therefore, it  follows from  Lemma (\ref{FKG-inequality-dim1}) that
\begin{align}
&\int_{\mathbb{R}^{\mathbb{N}}} \varphi([x|y]_{n+1})\psi([x|y]_{n+1}) d\mu^{[x|y]_{n+1}}_{n+1}(x) \nonumber \\
&\geq \int_{\mathbb{R}} \left(\int_{\mathbb{R}^{\mathbb{N}}}\varphi([x|t|y]_n)d\mu^{[x|t|y]_n}_n(x)\right)d\eta(t) \cdot \int_{\mathbb{R}} \left(\int_{\mathbb{R}^{\mathbb{N}}}\psi([x|t|y]_n) d\mu^{[x|t|y]_n}_n(x)\right)d\eta(t) \nonumber \,.
\end{align}

Furthermore, observe that the same procedure used above guarantees that
\[
\int_{\mathbb{R}} \left(\int_{\mathbb{R}^{\mathbb{N}}}\varphi([x|t|y]_n)d\mu^{[x|t|y]_n}_n(x)\right)d\eta(t) = \int_{\mathbb{R}^{\mathbb{N}}}\varphi([x|y]_{n+1})d\mu^{[x|y]_{n+1}}_{n+1}(x) \,,
\]

and the same conclusion is valid for $\psi$. Therefore,
\begin{align}
&\int_{\mathbb{R}^{\mathbb{N}}} \varphi([x|y]_{n+1})\psi([x|y]_{n+1}) d\mu^{[x|y]_{n+1}}_{n+1}(x) \nonumber \\
&\geq \int_{\mathbb{R}^{\mathbb{N}}}\varphi([x|y]_{n+1})d\mu^{[x|y]_{n+1}}_{n+1}(x) \cdot \int_{\mathbb{R}^{\mathbb{N}}}\psi([x|y]_{n+1}) d\mu^{[x|y]_{n+1}}_{n+1}(x) \nonumber \,.
\end{align}
\end{proof}

\section*{Acknowledgments}

The authors are grateful to the professor Leandro Cioletti for reading the first draft of this paper and his useful suggestions that improved our final version.


\begin{thebibliography}{21}



\bibitem{MR2864625}
A. T. Baraviera, L. M. Cioletti, A. O. Lopes, J. Mohr, and R. R. Souza.
\newblock On the general one-dimensional XY model: positive and zero temperature, selection and non-selection.
\newblock {\em  Rev. Math. Phys.}, 23(10):1063-1113, 2011.

\bibitem{MR2210682}
A. T. Baraviera, A. O. Lopes, and P. Thieullen.
\newblock  A large deviation principle for the equilibrium states of Holder
potentials: the zero temperature case.
\newblock {\em Stoch. Dyn.}, 6(1):77-96, 2006.

\bibitem{MR3227149}
R. Bissacot and R. Freire.
\newblock  On the existence of maximizing measures for irreducible countable Markov shifts: a dynamical proof.
\newblock {\em Ergodic Theory Dynam. Systems}, 34(4):1103-1115, 2014.

\bibitem{BMP15}
R. Bissacot, J. Mengue, and E. P\'erez.
\newblock  A Large Deviation Principle for Gibbs States on Countable Markov Shifts at Zero Temperature.
\newblock {\em https://arxiv.org/abs/1612.05831}, preprint.

\bibitem{MR1958608}
J. Br\'emont. 
\newblock Gibbs measures at temperature zero.
\newblock {\em Nonlinearity}, 16(2):419-426, 2003.

\bibitem{CiLo17}
L. Cioletti and A. O. Lopes.
\newblock Correlation inequalities and monotonicity properties of the Ruelle operator.
\newblock {\em Stoch. Dyn.}, to appear.

\bibitem{MR3690296}
L. Cioletti and A. O. Lopes.
\newblock Interactions, specifications, DLR probabilities and the Ruelle operator in the one-dimensional lattice.,
\newblock {\em Discrete Contin. Dyn. Syst.}, 37(12):6139-6152, 2017.

\bibitem{MR1942310}
D. Fiebig, U. Fiebig, and Y. Michiko.
\newblock Pressure and equilibrium states for countable state Markov shifts.
\newblock {\em Israel J. Math.}, 131:221-257, 2002.

\bibitem{MR3864383}
R. Freire and V. Vargas.
\newblock Equilibrium states and zero temperature limit on topologically transitive countable markov shifts.
\newblock {\em Trans. Amer. Math. Soc.}, Published electronically, 2018.

\bibitem{MR2151222}
O. Jenkinson, R. D. Mauldin, and M. Urba\'nski.
\newblock Zero temperature limits of Gibbs-equilibrium states for countable alphabet subshifts of finite type.
\newblock {\em J. Stat. Phys.}, 119(3-4):765-776, 2005.

\bibitem{MR2279266}
O. Jenkinson, R. D. Mauldin, and M. Urba\'nski.
\newblock Ergodic optimization for countable alphabet subshifts of finite type.
\newblock {\em Ergodic Theory Dynam. Systems}, 26(6):1791-1803, 2006.

\bibitem{MR2800665}
T. Kempton.
\newblock Zero temperature limits of Gibbs equilibrium states for countable Markov shifts.
\newblock {\em J. Stat. Phys.}, 143(4):795-806, 2011.

\bibitem{MR2176962}
R. Leplaideur.
\newblock A dynamical proof for the convergence of Gibbs measures at temperature zero.
\newblock {\em Nonlinearity}, 18(6):2847-2880, 2005.

\bibitem{MR3377291}
A. O. Lopes, J. K. Mengue, J. Mohr, and R. R. Souza.
\newblock Entropy and variational principle for one-dimensional lattice systems with a general a priori probability: positive and zero temperature.
\newblock {\em Ergodic Theory Dynam. Systems}, 35(6):1925-1961, 2015.

\bibitem{MR2496111}
A. O. Lopes, J. Mohr, R. R. Souza, and Ph. Thieullen.
\newblock Negative entropy, zero temperature and Markov chains on the interval.
\newblock {\em Bull. Braz. Math. Soc.}, 40(1):1-52, 2009.


\bibitem{MR2003772}
R. D. Mauldin and M. Urba\'nski.
\newblock {\em Graph directed Markov systems,}
\newblock Cambridge University Press, Cambridge, 2003.

\bibitem{Mohr} 
J. Mohr,
\newblock Product type potential on the XY model: selection of maximizing probability and a large deviation principle.
\newblock {\em https://arxiv.org/abs/1805.09858}, preprint.

\bibitem{MR1085356}
W. Parry and M. Pollicott.
\newblock Zeta functions and the periodic orbit structure of hyperbolic dynamics.
\newblock {\em Ast\'erisque}, vol. 187-188, 1990.

\bibitem{MR0234697}
D. Ruelle.
\newblock Statistical mechanics of a one-dimensional lattice gas.
\newblock {\em Comm. Math. Phys.}, 9: 267-278, 1968.

\bibitem{MR1738951}
O. M. Sarig.
\newblock Thermodynamic formalism for countable Markov shifts.
\newblock {\em Ergodic Theory Dynam. Systems}, 19(6):1565-1593, 1999.



\end{thebibliography}
\end{document}